\documentclass[12pt]{amsart}
\usepackage{amsmath, amssymb, amscd, mathrsfs, slashed, enumerate, url}
\usepackage{xcolor}
\usepackage[all,cmtip]{xy}
\usepackage{multicol}
\usepackage{indentfirst}
\usepackage{latexsym}
\usepackage{bm}
\usepackage{graphicx}
\usepackage{subfigure,esint}
\usepackage{float}
\usepackage{verbatim}
\usepackage{comment}
\usepackage[top=1in, bottom=1in, left=1.25in, right=1.25in]{geometry}
\usepackage{epsfig,dsfont,amsthm,amsfonts,amsbsy}
\usepackage[colorlinks]{hyperref}
\newcommand{\gpp}{\mathfrak{g}_P}

\newcommand{\MGC}{\mathcal{G}_{\mathbb{C}}}

\newcommand{\End}{\mathrm{End}\,}
\newcommand{\cpx}{\mathrm{cpx}}
\newcommand{\app}{\mathrm{app}}

\newcommand{\lan}{\langle }
\newcommand{\ran}{\rangle}

\newcommand{\ML}{\mathcal{L}}
\newcommand{\MD}{\mathcal{D}}

\newcommand{\MK}{\mathcal{K}}

\newcommand{\MM}{\mathcal{M}}
\newcommand{\MH}{\mathcal{H}}

\newcommand{\MO}{\mathcal{O}}

\newcommand{\Ker}{\mathrm{Ker}}

\newtheorem{theorem}{Theorem}[section]

\newtheorem{corollary}[theorem]{Corollary}

\newtheorem{definition}[theorem]{Definition}

\newtheorem{proposition}[theorem]{Proposition}
\newtheorem*{remark}{Remark}

\newcommand{\MC}{\mathcal{C}}
\newcommand{\Tr}{\mathrm{Tr}}

\newcommand{\st}{\star}
\newcommand{\we}{\wedge}
\newcommand{\pa}{\partial}

\newcommand{\RP}{\mathbb R^+}
\newcommand{\EBE}{extended Bogomolny equations\;}
\newcommand{\ti}{\times}

\newcommand{\Si}{\Sigma}
\newcommand{\bz}{\bar{z}}
\newcommand{\vp}{\varphi}

\newcommand{\ME}{\mathcal{E}}
\newcommand{\da}{\dagger}
\newcommand{\al}{\alpha}

\newcommand{\na}{\nabla}
\newcommand{\ep}{\epsilon}
\newcommand{\hA}{\widehat{A}}

\newcommand{\hP}{\widehat{\Phi}}

\newcommand{\be}{\beta}

\newcommand{\calC}{\mathcal C}

\newcommand{\calX}{\mathcal X}

\newcommand{\del}{\partial}
\newcommand{\RR}{\mathbb R}
\newcommand{\CC}{\mathbb C}

\newcommand{\lam}{\lambda}

\newcommand{\si}{\sigma}
\newcommand{\CST}{\mathbb{C}^{\times}}
\newcommand{\isu}{i\mathfrak{su}}
\newcommand{\mft}{\mathfrak{t}}
\newcommand{\mfe}{\mathfrak{e}}

\newcommand{\MP}{\mathcal{P}}

\newcommand{\diag}{\mathrm{diag}}
\newcommand{\mo}{\mathrm{mod}}

\newcommand{\MG}{\mathcal{G}}
\newcommand{\slf}{\mathfrak{sl}}

\newcommand{\su}{\mathfrak{su}}
\newcommand{\Lam}{\Lambda}
\newcommand{\SL}{\mathrm{SL}}

\newcommand{\Hit}{\mathrm{Hit}}
\newcommand{\ie}{\mathrm{ie}}
\newcommand{\Ad}{\operatorname{Ad}}
 
\newcommand{\vpq}{\vp_{\mathbf{q}}}
\newcommand{\vpz}{\vp_{\mathbf{0}}}
\newcommand{\MOP}{\MM_{\mathrm{Oper}}}
\newcommand{\GE}{\mathfrak{g}_E}
\newcommand{\TEBE}{twisted extended Bogomolny equations\;}
\newcommand{\DID}{\MD_i^{\da_H}}
\newcommand{\rank}{\mathrm{rank}}
\newcommand{\Op}{\mathrm{Oper}}
\newcommand{\Opers}{\mathrm{Oper}}
\newcommand{\rflat}{\mathrm{flat}}
\newcommand{\NHC}{\mathfrak{K}_{\mathrm{NAH}}}
\newcommand{\TBE}{\mathrm{TBE}}
\newcommand{\msA}{\mathscr{A}}
\newcommand{\msN}{\mathscr{N}}

\usepackage[utf8]{inputenc}
\usepackage[english]{babel}
\usepackage{fancyhdr}

\begin{document}
	\title[Opers and the twisted Bogomolny equations]{Opers and the twisted Bogomolny equations}
	\author{Siqi He} 
	\address{Simons Center for Geometry and Physics\\Stony Brook, NY, 11794 USA}
	\email{she@scgp.stonybrook.edu}
	\author{Rafe Mazzeo}
	\address{Department of Mathematics, Stanford University\\Stanford,CA 94305 USA}
	\email{rmazzeo@stanford.edu}
	
	\begin{abstract}
In this paper, we study the dimensionally reduced twisted Kapustin-Witten equations on the product of
a compact Riemann surface $\Si$ with $\RP_y$. The main result is a Kobayashi-Hitchin type correspondence between the
space of tilted Nahm pole solutions and the moduli space of Beilinson-Drinfeld opers. This corroborates a
prediction of Gaiotto and Witten \cite[p.971]{gaiotto2012knot}.
\end{abstract}
	
\maketitle

\medskip

\begin{section}{Introduction}
Let $(M,g)$ denote an oriented Riemannian $4$-manifold and $P$ a principal $SU(n)$ bundle over $M$ with adjoint
bundle $\gpp$.  The twisted Kapustin-Witten (TKW) equations \cite{KapustinWitten2006} are a one-parameter family
of equations, parametrized by $t \in (0,\infty)$, for a pair $(A, \Phi)$, where $A$ is a connection and $\Phi$ is a
$\gpp$-valued $1$-form: 
\begin{equation}
\begin{split}
& F_A-\Phi\we\Phi+\frac{t-t^{-1}}{2}d_A\Phi+\frac{t+t^{-1}}{2}\st d_A\Phi=0, \\
& d_A\st \Phi=0.
\end{split}
\label{TKW}
\end{equation}
When $t=1$, these equations have the particularly simple form 
\begin{equation}
\begin{split}
F_A-\Phi\we\Phi+\st d_A\Phi=0,\quad d_A\st\Phi=0.
\label{KW}
\end{split}
\end{equation}

One important case is when $M=X\ti\RP_y$, where $X$ is a $3$-manifold.  A fascinating proposal of Witten \cite{witten2011fivebranes}
interprets the Jones polynomial of knots in $X$ by counting solutions to \eqref{KW} which satisfy certain `Nahm pole' singularities at $y=0$,
see \cite{gaiotto2012knot,Witten2014LecturesJonesPolynomial,Witten2016LecturesGaugeTheory} for a more detailed explanation, along with \cite{MazzeoWitten2013,MazzeoWitten2017,He2017,Taubescompactness,LeungRyosuke2018} for analytic theory related to this program.
A similar program using the TKW equations \eqref{TKW} to approach the Jones polynomial is also discussed in
\cite{witten2011fivebranes, gaiotto2012knot}; a clearer formulation appears in \cite{mikhaylov2018teichmuller}, where the singular
boundary conditions appropriate for these equations when $t \neq 1$ are called the tilted Nahm pole boundary condition. We describe
these below. 

We consider here the dimensionally reduced TKW equations on $\Si \times \RP_y$, where $\Si$ is a compact Riemann surface.   These are
the TKW equations on $S^1 \times \Si \ti \RP_y$ for fields which are invariant in the $S^1$ direction.   These fields consist of a connection
$A$, a $\gpp$-valued $1$-form $\phi$ and $\gpp$-valued $0$-forms $A_1$ and $\phi_1$; the corresponding `twisted Bogomolny equations'
take the form 
\begin{equation}
\begin{split}
&F_A-\phi\we\phi+\frac{t-t^{-1}}{2}d_A\phi-\frac{t+t^{-1}}{2}(\st d_A\phi_1+\st[\phi,A_1])=0,\\
&d_AA_1-[\phi,\phi_1]+\frac{t-t^{-1}}{2}(d_A\phi_1+[\phi,A_1])-\frac{t+t^{-1}}{2}\st d_A\phi=0,\\
&d_A^{\st}\phi-[\phi_1,A_1]=0.
\label{KW2}
\end{split}
\end{equation}

Write $t=\tan(\frac{\pi}{4}-\frac{3}{2}\beta)$.  It was observed by Gaiotto and Witten in \cite{gaiotto2012knot} that with the assumption
$A_1-\tan\beta\phi_1=0$, the twisted Bogomolny equations have a Hermitian-Yang-Mills structure, leading them to conjecture 
that there should be a Donaldson-Uhlenbeck-Yau type theorem in this setting.   Such a result is now fully understood in the special
case $\beta = 0$ ($t=1$), cf.\ \cite{HeMazzeo2017,HeMazzeo2018,HeMazzeo2019}. The results in those two papers give a precise
correspondence in this spirit between flat $\SL(2,\CC)$ connections over $\Si$ and solutions of the \EBE converging to these connections as $y \to \infty$
and satisfying certain singular boundary conditions at $y=0$.  More precisely, a flat $\SL(2,\RR)$ connection in the Hitchin section at
infinity corresponds to solutions of the \EBE satisfying the Nahm pole boundary conditions at $y=0$; an arbitrary stable Higgs pair (equivalently,
a flat $\SL(2,\CC)$ connection), together with a holomorphic line subbundle, corresonds to a solution of the \EBE satisfying the Nahm pole
boundary conditions with extra singularities along a divisor determined by the line bundle and the Higgs field.   It is the generalization of the
former of these two theorems which is the subject of this paper;  a full generalization awaits a better understanding of the knot singularities
for the tilted Nahm pole boundary conditions.

The study of \eqref{KW2} when $\beta\neq 0$ $(t\neq 1)$ is motivated by the Atiyah-Floer approach to the Kapustin-Witten equations, see
\cite{ManolescuAbouzaid17,ManolescuCote18} for recent progress. As pointed out in \cite[Section 3]{gaiotto2012knot}, this
Atiyah-Floer approach has physical obstructions and is unstable for the equations \eqref{KW} when $t=1$, indicating that it may not be
possible to recover the Jones polynomial entirely from that specialization of the equations. The paper \cite{Gukov2017bps} contains a more
detailed explanation of this.

In any case, we consider here only the cases where $\beta\neq 0$.  Denote by $\MM_{\TBE}^{\beta}$ the space of solutions to the
twisted Bogomolny equations with gauge group $SU(n)$ and with tilted Nahm pole conditions at $y=0$ and a certain boundary condition
to be explained later as $y \to \infty$.  We denote by $\MM_{\Opers}^{\beta}$ the twisted oper moduli space with parameter $\tan\beta$; this
is diffeomorphic to the usual oper moduli space of Beilinson-Drinfield \cite{beilinson2005opers}, and is defined in Section
\ref{Sec_twistedopers}. Using the Hermitian-Yang-Mills structure, Gaiotto and Witten \cite{gaiotto2012knot} define
$$
I^{\beta}_{\Opers}:\MM_{\TBE}^{\beta}\to \MM_{\Opers}^{\beta},
$$
explained in Section \ref{Sec_KobayashiHitchinMap} below, and predict that it is a bijection.  We confirm their prediction here. 
\begin{theorem}
The map $I^{\beta}_{\Opers}$ is a bijection when the genus $g$ of $\Si$ is greater than $1$. 
\begin{itemize} 
\item[i)] For each element in $\MM_{\Opers}^{\beta}$, there exists a solution
to \eqref{KW2} with tilted Nahm pole singularity at $y=0$; 
\item[ii)] If two solutions satisfy tilted Nahm pole boundary condition and have the same image by $I^{\beta}_{\Opers}$, then they
are gauge equivalent.
\end{itemize}
\end{theorem}
	
There is an identification of $\MM^{\beta}_{\Opers}$ with $\oplus_{i=2}^nH^0(K^{i})$, where $K$ is the canonical bundle over $\Si$, which
gives a topology and differential structure to this space. 
\begin{theorem}
The map $I_{\Opers}^{\beta}$ is a diffeomorphism.
\end{theorem}
	
\textbf{Acknowledgements.} The first author thanks Victor Mikhaylov for assistance with the computations in the Appendix and Brian Collier,
Simon Donaldson, Ciprian Manolescu and Du Pei for helpful discussions.  R.M.\ has been supported by the NSF grant DMS-1608223.
\end{section}

\begin{section}{The Twisted Extended Bogomolny Equations}
\begin{subsection}{Hermitian Geometry for the Twisted Bogomolony Equations}
  We write the product metric on $\Si\ti\RP$, where $\Si$ is a compact Riemann surface with genus $g\geq 1$, as $g=g_0^2|dz|^2+dy^2$.
  Let $E$ be a rank $n$ complex vector bundle on this space with $\det E = 0$. An $SU(n)$ structure on $E$ is determined by a Hermitian
  metric $H$. The adjoint bundle is denoted $\GE$. 
	
  Let $A$ be a connection on $E$ and suppose that $\phi\in \Omega^1(\GE)$ and $\phi_1\in\Omega^0(\GE)$. In a unitary gauge defined
  by $H$,  these satisfy $A^{\st}=-A,\;\phi^{\st}=-\phi,\;\phi_1^{\st}=-\phi_1$ where $\st$ is the conjugate transpose defined by $H$.
  Gaiotto and Witten observe in \cite{gaiotto2012knot} that the \TEBE have a Hermitian Yang-Mills structure, and hence there should
  be a Donaldson-Uhlenbeck-Yau type result as in \cite{donaldson1985anti, uhlenbeck1986existence}.  Write $z=z_2+iz_3$ for a local holomorphic
coordinate on $\Si$ and $y$ for the linear coordinate on $\RP$. Then 
\[
d_A=\na_2dx_2+\na_3dx_3+\MD_ydy,\ \mbox{and}\ \ \phi=\phi_2dx_2+\phi_3dx_3 =\frac{1}{2}(\phi_z dz+\phi_{\bz} d\bz).
\]

Following \cite{gaiotto2012knot,witten2011fivebranes}, define
\begin{equation}
\begin{split}
\MD_1=(D_{\bz}-\phi_{\bz}\tan\beta)d\bz,\;
\MD_2=(D_z+\phi_z\cot \beta)dz,\;
\MD_3=D_y-i\frac{\phi_1}{\cos \beta}
\end{split}
\end{equation}
and their adjoints 
\begin{equation}
\begin{split}
\MD_1^{\da}=(D_{z}-\phi_{z}\tan\beta)dz,\;
\MD_2^{\da}=(D_{\bz}+\phi_{\bz}\cot \beta)d\bz,\;
\MD_3^{\da}=D_y+i\frac{\phi_1}{\cos \beta}.
\end{split}
\end{equation}
Using this, the \TEBE can be written in the particularly elegant form
\begin{equation}
\begin{split}
&[\MD_i,\;\MD_j]=0,\;\;i,j=1,2,3,\\
&\frac{i}{2}\Lambda(\cos^2 \beta[\MD_1,\MD_1^{\da}]-\sin^2\beta[\MD_2,\MD_2^{\da}])+[\MD_3,\MD_3^{\da}]=0,
\label{TEBE}
\end{split}
\end{equation}
where $\Lambda:\Omega^{1,1}\to\Omega^0$ is inner product with the K\"ahler form $\frac{i}{2}dz\we d\bz$.
	
Consider the group $\MGC$ of complex gauge transformations of $E$, and set $\MG:=\{g\in\MGC|\, gg^{\st}=1\}$ (where $g^{\st}$ is the conjugate
transpose of $g$), the subspace of unitary gauge transformations, preserving the Hermitian metric $H$.  Any $g\in\MGC$ acts by
\begin{equation*}
\MD_i^{g}:=g^{-1} \circ \MD_i\circ g,\;\MD_i^{\da,g}:=g^{\st}\circ\MD_i\circ (g^{-1})^{\st}.
\end{equation*}

The equations $[\MD_i,\MD_j]=0$ are gauge invariant under the full complex gauge group, while the full system is only invariant
under $\MG$. The second equation in \eqref{TEBE} may be considered as the moment map for the first set of equations.
\end{subsection}

\begin{subsection}{Holomorphic data}
As in Donaldson and Uhlenbeck-Yau \cite{donaldson1985anti,uhlenbeck1986existence, Donaldson1987Infinite}, we begin by interpreting
the geometric meaning of a set of holomorphic data satisfying the $\MGC$ invariant part of the equations. The main theorem is proved later by
finding a solution compatible with any choice of holomorphic data which also solves the other moment map equation. 

Let $E_y$ be the restriction of $E$ to the slice $\Si\times\{y\}$, and consider the $SL(n,\mathbb{C})$ connection $\MD_{\Si_y}=\MD_1+\MD_2$
on $E_y$. The commutation relation $[\MD_1,\MD_2]=0$ is equivalent to the flatness of $\MD_{\Si_y}$, i.e., $\MD_{\Si_y}^2=0$.
Furthermore, $\MD_3$ is a covariant derivative in the $y$ direction, and hence defines a parallel transport in
this `vertical' direction.  The commutation relations $[\MD_1,\MD_3]=[\MD_2,\MD_3]=0$ identify the flat connections $\MD_{\Si_y}$ over the
different slices $E_y$ by parallel transport.
	
Based on this, we define the holomorphic data as a rank $n$ bundle $E$ with $\det(E)=0$, and a system of operators $\Theta=(\MD_1,\MD_2,\MD_3)$
acting on sections of $E$ such that for any smooth function $f$ and section $s$ of $E$, we have
\begin{itemize}
\item $\MD_1(fs) = \bar{\pa}f s + f \MD_1 s$, \;$\MD_2(fs) = \pa f s + f \MD_1 s$,\;$\MD_3(fs) = (\del_{y}f) s + f \MD_3 s$,
\item $[\MD_i,\;\MD_j]=0$ for all $i, j$. 
\end{itemize}
The complex gauge group $\MGC$ acts on $(E,\Theta)$ by $\MD_i\to g^{-1}\circ\MD_i\circ g$.    As above, parallel transport yields the identification
\begin{equation}
\{(E,\Theta)\}/\MGC \cong \{ \mbox{flat connections over} \ \Si\}/\MGC.
\label{generalcomplexstructurequotientGC}
\end{equation}
	
Rather than letting $\MGC$ act on the data set $(E, \Theta)$, it is easier to fix $\Theta$ and let $\MGC$ act on the Hermitian metric.
We thus regard the real moment map in \eqref{TEBE} as an equation for the Hermitian metric. Given $(E,\Theta)$,
a Hermitian metric determines the adjoints $\DID$ of the operators $\MD_i$ by
\begin{itemize}
\item $\bar{\pa}(H(s,s'))=H(\MD_1 s,s')+H(s,\MD_1^{\da_H}s')$,
\item ${\pa}(H(s,s'))=H(\MD_2 s,s')+H(s,\MD_2^{\da_H}s')$,
\item $\pa_{y}(H(s,s'))=H(\MD_3 s,s')+H(s,\MD_3^{\da_H}s')$.
\end{itemize}
Now define the unitary operators and forms
\begin{equation*}
\begin{split}
&\MD_z=\sin^2\beta \, \MD_2+\cos^2\beta\, \MD_1^{\da_H},\;\MD_{\bz}=\cos^2\beta\, \MD_1+\sin^2\beta\, \MD_2^{\da_H},\\
&\phi_z=\sin\beta\cos\beta\, (\MD_2-\MD_1^{\da_H}),\;\phi_{\bz}=\sin\beta\cos\beta\, (\MD_2^{\da_H}-\MD_1),\\
&\MD_y=\frac{\MD_3+\MD_3^{\da_H}}{2},\;\phi_1=\frac{i\cos\beta}{2}\, (\MD_3-\MD_3^{\da_H}).
\end{split}
\end{equation*}
We obtain from these a unitary connection $\na_A:=\MD_z+\MD_{\bz}+\MD_y$ and a triple $(A,\phi,\phi_1)$
called the Chern connection of $(E,\Theta,H)$. 
	
We can express these operators in different gauges. A gauge is called parallel holomorphic if in this trivialization, $\MD_1=\bar{\pa}$
and $\MD_3 = \del_y$, and unitary if in this trivialization the matrices $(A,\phi,\phi_1)$ are unitary.  As in \cite{atiyah1978geometry}, we
record how to relate these two gauge choices. 
	
\begin{proposition}{\cite{HeMazzeo2017}} \label{complexgaugeactionchange}
Given $(E,\Theta, H)$ as above, there is a unique triplet $(A,\phi,\phi_y)$ compatible with the unitary and holomorphic structures. 
In other words, in any unitary gauge, $A^{\st}=-A$, $\phi^{\st}=\phi$, $\phi_1^{\st}=-\phi_1$, while in every parallel holomorphic gauge,
$\MD_1=\overline{\partial}_E$ and $\MD_3=\pa_y$, i.e., $A^{(0,1)}= A_y-i\phi_1=0$.
\end{proposition}
In a local holomorphic trivialization of $E$, we represent the metric by a Hermitian matrix which we also denote by $H$.
Write $\MD_1=\bar{\pa}$, $\MD_2=\pa+\al$, $H=g^{\da}g$ for $g=H^{\frac{1}{2}} \in\MGC$.  In holomorphic gauge, 
$$
\MD_1=\bar{\pa},\;\MD_1^{\da}=\pa+H^{-1}\pa H =\pa+g^{-1}(g^{\da})^{-1}(\pa_z g^{\da}) g+g^{-1} \pa_z g.
$$
Thus $g$ transforms from holomorphic to unitary gauge. In unitary gauge, we write these operators as $\MD_i^U$ and $(\MD_i^{\da})^U$,
and then have
\begin{equation}
\begin{split}
&\MD_1^U=\bar{\pa}- (\bar{\pa}g) g^{-1},\;(\MD_1^{\da})^U=\pa+(g^{\da})^{-1}\pa_z g^{\da},\\
&\MD_2^U=\pa+g\al g^{-1}-(\pa g) g^{-1},\;(\MD_2^{\da})^U=\bar{\pa}+(g^{\da})^{-1}\bar{\pa}g^{\da}-(g^{\da})^{-1}\bar{\al}^{T}g^{\da},\\
&\MD_3^U=\pa_y-(\pa_yg) g^{-1},\;(\MD_3^{\da})^{U}=\pa_y+(g^{\da})^{-1}\pa_yg^{\da}.
\end{split}
\end{equation}
	
Thus in unitary gauge, 
\begin{equation}
\begin{split}
&A_z=\sin^2\beta(g\al g^{-1}-(\pa g) g^{-1})+\cos^2\beta (g^{\da})^{-1}\pa_zg^{\da},\;A_{\bz}=-A_z^{\da},\\
&\phi_z=\sin\beta\cos\beta(g\al g^{-1}-(\pa g) g^{-1}-(g^{\da})^{-1}\pa_zg^{\da}),\;\phi_{\bz}=-(\phi_z)^{\da},\\
&A_y=\frac{1}{2}(- (\pa_yg) g^{-1}+(g^{\da})^{-1}\pa_y g),\;\phi_1=\frac{i\cos\beta}{2}(-(\pa_yg) g^{-1}-(g^{\da})^{-1}\pa_yg^{\da}).
\label{unitarygaugeformula}
\end{split}
\end{equation}
	
\end{subsection}
\end{section}
	
\begin{section}{Higgs bundles, the Teichm\"uller component and the space of opers}
	\label{Sec_twistedopers}
\subsection{The nonabelian Hodge correspondence}
\subsubsection{The de Rham moduli space}
As always, assume that $E$ is a bundle of rank $n$ with $\det(E)=0$ over the Riemann surface $\Si$ and $\na$ a connection on $E$.
This connection is irreducible if there is no parallel subbundle and completely reducible if every $\na$-invariant subbundle
has an $\na$-invariant complement. If $\MC_{\rflat}$ denotes the space of flat connections, then its tangent space $T_\na\MC_{\rflat}$
consists of the set of $\sigma\in \Omega^1(\End(E))$ such that $\na+\si$ is flat to first order, i.e., $T_\na \MC_{\rflat}:
=\{\si\in\Omega^1(\mathrm{End}(E))|\ \na \si=0\}$. The complex structure on $\SL(n,\CC)$ induces a complex structure $J$ on
$\MC_{\rflat}$, $J(\si)=i\si$.
	
Complex gauge transformations $g(\na):=g^{-1}\circ \na\circ g$ are the smooth automorphisms of $E$ acting trivially on
$\det(E)$. The de Rham moduli space of flat connections is 
$$
\MM_{\rflat}:=\{\na\in\MC_{\rflat}|\ \na \ \mbox{is completely reducible}\}/\MGC.
$$
The complex structure $J$ is preserved by the $\MGC$ action, and hence induces a complex structure on $\MM_{\rflat}$. 
	
\subsubsection{The Higgs bundle moduli space }	
Recall that an $\SL(n,\mathbb{C})$-Higgs bundle over the Riemann surface $\Si$
is a pair $(\ME,\vp)$, where $\ME$ is a vector bundle $E$ with holomorphic structure $\bar{\pa}_{E}$ and $\vp$ is a traceless
holomorphic $(1,0)$-form, $\bar{\pa}_E \vp=0$. We restrict to Higgs bundles with $\det(\ME)=0$. If $\MH$ denotes the set of
$\SL(n,\mathbb{C})$-Higgs bundles on $E$, then 
$$
T_{(\bar{\pa}_E,\vp)}\MH=\{(a,b)\in\Omega^{0,1}(\End E)\oplus \Omega^{1,0}(\End E)|\ \bar{\pa}_E b+[\vp,a]=0\}.
$$
A complex structure $I$ on $\MH$ is defined by $I(a,b)=(ia,ib)$, and the complex gauge group action is
$g(\ME,\vp):=(g^{-1}\circ \bar{\pa}_E\circ g,g^{-1}\circ\Phi\circ g)$.

A Higgs pair $(\ME,\vp)$ with $\deg \ME = 0$ is called stable if 
$\deg (V)<0$ for any holomorphic subbundle $V$ with $\varphi(V)\subset V\otimes K$, and polystable if it is a direct sum 
of stable Higgs pairs. We define the Higgs bundle moduli space $\MM_{\mathrm{Higgs}}$ by 
$$
\MM_{\mathrm{Higgs}}:=\{(\ME,\vp)\in\MH|\ (\ME,\vp)\ \mbox{is polystable}\}/\MGC.
$$ 
	
One of the most important features is the Hitchin fibration map 
\begin{equation}
\begin{split}
&\pi:\MM_{\mathrm{Higgs}}\to \oplus_{i=2}^n H^0(\Si,\;K^{i})\\
&\pi(\vp)=(p_2(\vp),\;\cdots,\;p_{n}(\vp)),
\label{HitchinFiberationMap}
\end{split}
\end{equation}
where $\det (\lambda -\vp) = \sum \lambda^{n-j} (-1)^j p_{j}(\vp)$. By \cite{hitchin1987stable}, this map
is proper.

\subsubsection{The nonabelian Hodge correspondence}
Given a Higgs bundle $(\ME,\vp)$ the Hitchin equations are equations for the Hermitian metric $H$
\begin{equation}
\begin{split}
F_H+[\vp,\vp^{\da_H}]=0,\;\bar{\pa}_{E}\vp=0,
\label{Hitchinequation}
\end{split}
\end{equation}
where $F_H$ is the curvature of the Chern-connection, $\vp^{\da}$ is the conjugation of $\vp$ w.r.t the Hermitian metric $H$.

\begin{theorem}{\cite{Hitchin1987Selfdual,Corlette88}}
For any Higgs pair $(\ME,\vp)$ on $\Si$, there exists an irreducible solution $H$ to the Hitchin equations 
if and only if this pair is stable, and a reducible solution if and only if it is polystable.
\end{theorem}
	
To any solution $H$ to the Hitchin equations is associated a flat $\mathrm{SL}(n,\CC)$ connection
$D = \nabla_H + \vp + \vp^{\da_H}$, and hence a representation $\rho: \pi_1(\Si) \to \mathrm{SL}(n,\CC)$,
which is well-defined up to conjugation.  Irreducibility of the solution is the same as irreducibility of the representation,
while reducibility corresponds to the fact that $\rho$ is reductive.   The map from flat connections back
to solutions of the Hitchin system involves finding a {\em harmonic metric} which yields a decomposition of $D = D^{\mathrm{skew}} 
+ D^{\mathrm{Herm}}$ into skew-Hermitian and Hermitian parts, so that $D^{\mathrm{Herm}} = \vp + \vp^{\star_H}$
and $( (D^{\mathrm{skew}})^{0,1}, \vp)$ satisfies Hitchin equations. The culmination of the work of Hitchin, Donaldson, Simpson and Corlette is the diffeomorphic equivalence between
the spaces of stable Higgs pairs, irreducible solutions of the Hitchin equations and irreducible flat connections; there is a similar equivalence for the polystable/reducible spaces. 

In summary, the non-abelian Hodge correspondence says that the map
\begin{equation}
\begin{split}
\NHC&:\MM_{\mathrm{Higgs}}\to\MM_{\rflat},\\
\NHC&(\ME,\vp)=A_H+\vp+\vp^{\da_H},
\end{split}
\end{equation}
is a diffeomorphism where $H$ solves the Hitchin equations and $A_H$ is the Chern-connection of $\ME$ and $H$.
	
\subsubsection{Hyperk\"ahler structure}
The moduli space $\MM_{\mathrm{Higgs}}$ has important geometric structure. The tangent space of $\MM_{\mathrm{Higgs}}$ at each point is a subspace of
$\Omega^{0,1}(\End E)\oplus \Omega^{1,0}(\End E)$ at that point, so it carries a hyperk\"ahler structure as in \cite{Hitchin1987Selfdual}:
ifet $a\in\Omega^{0,1}(\End E)$ and $b\in\Omega^{1,0}(\End E)$, then 
\begin{equation}
\begin{split}
I(a,b)=(ia,ib),\;J(a,b)=(ia^{\st},-ib^{\st}),\;K(a,b)=(-a^{\st},b^{\st}).
\label{hyperkahlerIJK}
\end{split}
\end{equation}

\subsection{The Hitchin component and opers}
In this subsection, we introduce the Hitchin component and the moduli space of opers.

\subsubsection{$\SL(n)$-Hitchin component}
Choose a spin structure $K^{\frac{1}{2}}$; then, for any $\mathbf{q}:=(q_2,\cdots,q_{n})$, define the Higgs bundle $(\ME, \vp_{\mathbf{q}})$ by
\begin{equation}
\begin{aligned}
\ME: & =S^{n-1}(K^{-\frac{1}{2}}\oplus K^{\frac{1}{2}})=K^{-\frac{n-1}{2}}\oplus K^{-\frac{n-1}{2}+1}\oplus\cdots\oplus K^{\frac{n-1}{2}} \\[5mm]
\vp_{\mathbf{q}} & =\begin{pmatrix}
0 & \sqrt{B_1} & 0 &\cdots& 0\\
0 & 0 & \sqrt{B_2} & \cdots& 0\\
\vdots & \vdots &\ddots & &\vdots\\
0 & \vdots & &\ddots &\sqrt{B_{n-1}}\\
q_{n}& q_{n-1} & \cdots & q_2 &0
\end{pmatrix},
\end{aligned}
\label{HitchincomponentHiggsfield}
\end{equation}
where $B_{i}=i(n-i)$ and $\sqrt{B_i}$ in the $(i,i+1)$ entry represents that multiple of the natural isomorphism $K^{-\frac{n-1}{2} + i  } \to 
K^{-\frac{n-1}{2} + i -1 }\otimes K$ and the maps along the bottom row are $H^0(\Si,K^{n-i}) \ni q_{n-i} :K^{- \frac{n-1}{2} + i }\to K^{ \frac{n-1}{2}} \otimes K$. 

The complex gauge orbit of this family of Higgs bundle,
\begin{equation}
\{(\ME:=S^{n-1}(K^{-\frac{1}{2}}\oplus K^{\frac{1}{2}}),\ \vp\ \mbox{as in}\ \eqref{HitchincomponentHiggsfield})\}/\MGC.
\end{equation}
is called the Hitchin component (or the Hitchin section), and denoted $\MM_{\mathrm{Hit}}$. Note that when $n$ 
is odd, only even powers of $K^{1/2}$ appear, so $\MM_{\mathrm{Hit}}$ is independent of the choice of spin structure in that case. 
We consider the map $p_{\Hit}$ defined as
\begin{equation}
\begin{split}
p_{\Hit}&: \mathcal B_{\mathrm{Hit}} := \oplus_{i=2}^{n}H^0(\Si,K^{i})\to \MM_{\mathrm{Higgs}},\\
p_{\Hit}&(q_2,\cdots,q_{n}):=\left[(\ME,\vp_{\mathbf{q}})\right],
\label{Hitchinsectionmap}
\end{split}
\end{equation}
with $(\ME,\vp_{\mathbf{q}})$ as in \eqref{HitchincomponentHiggsfield}. The space $\mathcal B_{\mathrm{Hit}}$ is called
the Hitchin base, and the image of $p_{\Hit}$ is the aforementioned Hitchin component $\MM_{\mathrm{Hit}}$.

\begin{theorem}{\cite{hitchin1992lie}}
 \begin{itemize}
\item[(1)] Every element in $\MM_{\mathrm{Hit}}$ is a stable Higgs pair, and $\MM_{\mathrm{Hit}}$ is parametrized by the vector space 
$\oplus_{i=2}^{n}H^0(\Si,K^{i})$;
\item[(2)] the map assigning to each element of $\oplus_{i=2}^{n}H^0(\Si,K^{i})$ the unique corresponding solution of the Hitchin equations
 is an equivalence from this vector space to one component of the moduli space of flat completely irreducible $\SL(n,\mathbb{R})$
  connections. The restriction of the Hitchin fibration map \eqref{HitchinFiberationMap}
$$
\pi|_{\MM_{\mathrm{Hit}}}:\MM_{\mathrm{Hit}}\to  \mathcal B_{\mathrm{Hit}} 
$$ 
is a diffeomorphism. 
\item[(3)] $p_{\Hit}$ is a holomorphic embedding into $\MM_{\mathrm{Higgs}}$ with respect to the holomorphic structure $I$.
\end{itemize}
\end{theorem}

\subsubsection{Opers}
We begin with the definition of opers from \cite{beilinson2005opers}.
\begin{definition}
\label{Def_Opers}
An $\SL(n)$-oper on a Riemann surface $\Si$ is a triple $(E,F_\bullet,\na)$, where $E$ is a rank $n$ holomorphic bundle
with $\det(E)=0$, $\na$ is a flat connection and $F_\bullet$ is a complete filtration of $E$ by holomorphic subbundles, 
$0=F_0\subset F_1\subset \cdots \subset F_n=E$, satisfying
\begin{itemize}
\item for any section $s$ of $F_i$, $\nabla s$ is a section of $F_{i+1}\otimes K$;
\item the induced map $\na:F_{i}/F_{i-1}\to F_{i+1}/F_i\otimes K$ is an isomorphism of line bundles, $i = 1, \ldots, n-1$.
	\end{itemize}
\end{definition}

The complex gauge group acts naturally on the space of opers and we define the oper moduli space
$$
\MOP:=\{(E,F_\bullet,\nabla)\}/\MGC.
$$
\begin{proposition}{\cite{wentworth2014higgs}}
If the genus of $\Si$ is greater than $1$, then 
\begin{itemize} 
\item[i)] the holonomy representation of an $\SL(n)$-oper is irreducible;
\item[ii)] setting $\ML\cong E/F_{n-1} $, then $det(F_j)\cong \ML^j\otimes K^{\frac{nj-j(j+1)}{2}}$ and $\ML^n\cong \MK^{\frac{-n(n-1)}{2}}$;
\item[iii)] the oper structure on $E$ is uniquely determined by $\ML$; in particular, the isomorphism class of $\ML$ is
fixed on every connected component of $\MOP$. 
\end{itemize}
\end{proposition}
\begin{remark}
  There are precisely $n^{2g}$ possibilities for the line bundle $\ML$, corresponding to the $n^{2g}$ ways of lifting a
  monodromy representation in $\mathrm{PSL}(n,\mathbb{C})$ to $\SL(n,\mathbb{C})$; these choices label the different
  components of $\MOP$. We fix one such choice and take $\ML =K^{-(n-1)/2}$.
\end{remark}

The next result identifies the Hitchin section with the space of opers.

\begin{theorem}{\cite{beilinson2005opers}}
  \label{ClassificationofOpers}
Consider a Higgs bundle $(\ME=S^{n-1}(K^{-\frac{1}{2}}\oplus K^{\frac{1}{2}}),\vpq)\in\MM_{\mathrm{Hit}}$. Let $h_{\mathbf{q}}$ be the
solution to the Hitchin equation \eqref{Hitchinequation} for the Higgs field $(\ME,\vp_{\mathbf{q}})$ with $\mathbf{q} \in \mathcal B_{\mathrm{Hit}}$
and Chern connection $D_{h_{\mathbf{q}}}$. Write $F_0=0$, $F_i= K^{-\frac{n-1}{2}}\oplus\cdots\oplus K^{-\frac{n-1}{2}+i-1}$ for $1\leq i\leq n$ and $F_{\bullet}$ 
the associated filtration.  Define $\na_{\mathbf{q}} = D_{h_{\mathbf{0}}} + \vpq + \vp_{\mathbf{q}}^{\dag_{h_{\mathbf{0}}}}$.
      Then $\MM_{\Opers}\cong \{\nabla_{\mathbf{q}}\}/\MGC$. In other words, this family of flat connections exhausts the entire space of opers. 
\end{theorem}



\subsection{Twistings}
Both the Hitchin equations and opers may be `twisted' by a nonzero complex number $w$.  These twisted analogues are equivalent
to the untwisted ones by a constant gauge transformation, but these are the objects most naturally compatible with the twisted \EBE.

\subsubsection{Twisted opers}
Define $\mathbf{q}_w =(w^2q_2,w^3q_3,\cdots,w^nq_n)$ where $\mathbf{q} \in \mathcal B_{\mathrm{Hit}}$ and
$w\in\mathbb{C}^{\times}$. 
Note that $\vp_{\mathbf{q}}$ is conjugated to $\vp_{\mathbf{q}_w}$ up to a constant by the constant gauge transformation
$g_w=\diag(w^{\frac{n-1}{2}},w^{\frac{n-3}{2}},\cdots, w^{-\frac{n-1}{2}})$.
The associated oper is \begin{equation}
\nabla^w_{\mathbf{q}} = g_w^{-1}\nabla_{\mathbf{q}}g_w=\nabla_{\mathbf{q}_w}=D_{h_{\mathbf{0}}}+w^{-1}\vpq+w\vpq^{\da_{h_{\mathbf{0}}}}.
\label{Operconnectionw}
\end{equation}  It follows directly
from Theorem \ref{ClassificationofOpers} that for any oper $\nabla$ and any $w \in \CC^\times$, there exists a unique $\mathbf{q}$
such that $\nabla=\na_{\mathbf{q}}^w$.     Just as for the case $w=1$, we can define the twisted oper moduli space
$\mathcal M^w_{\Opers}$ carefully below.

\subsubsection{The twisted Hitchin equations}
Let $(\ME,\vp,H)$ be a solution of the Hitchin equations. 
Defining $\MP_1:=\MD_{\bz}d\bz$, $\MP_2=\vp=i\phi_zdz$, we can rewrite the Hitchin equations as 
\begin{equation}
[\MP_1,\MP_2]=0,\;\Lambda([\MP_1,\MP_1^{\da_H}]+[\MP_2,\MP_2^{\da_H}])=0.
\end{equation}

The map $\NHC:\MM_{\rm{Higgs}}\to\MM_{\rflat}$ introduced earlier associates to any stable Higgs bundle $(\ME, \vp)$ the
flat connection $\na_A + \vp+\vp^{\da_H}$, where $H$ is the solution of the Hitchin equations and  $A_H$ is its Chern connection $A_H$.
However,  if $w\in\CST$,  we can define the twisted operators 
\[
\MP_1^{w}:=(\MD_{\bz}-w\phi_{\bz})d\bz,\ \ \MP_2^w:=(\MD_z+w^{-1}\phi_z)dz,
\]
where $d_A=D_zdz+D_{\bz}d\bz$ and $\phi=\phi_zdz+\phi_{\bz}d\bz$. Then $\na^w =\MP_1^w+\MP_2^w$ is also a flat connection.
This gives a family of maps: 
$$
\NHC^w:\MM_{\mathrm{Higgs}}\to\MM_{\rflat}, \qquad w \in \CST.
$$
	
The twisted operators $\MP_1^w,\;\MP_2^w$ lead to a new system of $w$-twisted Hitchin equations 
\begin{equation}
[\MP_1^{w},\MP_2^w]=0,\qquad \Lambda([\MP_1^{w},(\MP_1^w)^{\da_H}]-|w|^2[\MP_2^{w},(\MP_2^{w})^{\da_H}])=0.
\label{CstarHitchinequation}
\end{equation}
As before, given a flat connection $\na:=\MP_1^w+\MP_2^w$, \eqref{CstarHitchinequation} is an equation for the Hermitian metric,
and is equivalent to the standard untwisted Hitchin equations.
		
\begin{proposition}
	Define the $w$-twisted oper moduli space 
	\begin{equation}
	\begin{split}
	\MM^w_{\Opers}:=\{(\nabla,H)|\ \nabla \in \MM_{\Opers},\;H\;solves\;\eqref{CstarHitchinequation}\}/\MG,
	\end{split}
	\end{equation}
	\begin{itemize}
		\item [i)] For any $\na\in\MM_{\Opers}$, there exists a unique $H$ solving
		\eqref{CstarHitchinequation}; 
		\item [ii)] $\MM^w_{\Opers}$ is diffeomorphic to $\MM_{\Opers}$; 
		\item [iii)]\cite[Section 3]{KapustinWitten2006} $\MM^w_{\Opers}$ is a holomorphic symplectic submanifold with respect to the complex structure
		\begin{equation}
		I_w=\frac{1-\bar{w}w}{1+\bar{w}w}I+\frac{i(w-\bar{w})}{1+\bar{w}w}J+\frac{w+\bar{w}}{1+\bar{w}w}K.
		\end{equation}
	\end{itemize}
\end{proposition}
\begin{proof}
	For i), existence and uniqueness follows from Proposition \ref{solutioncstarhitchin} and \ref{Prop_uniqueness}, and ii) follows directly from  i).
	For iii), We can understand opers in terms of the hyperk\"ahler structure, cf.\ \cite[Section 3]{KapustinWitten2006}.
	Fixing $w\in \CST$, choose the holomorphic coordinates $A_z+w^{-1}\phi_z,\;A_{\bz}-w\phi_{\bz}$  for $\MM^w_{\Opers}$.
	Then $\MM^w_{\Opers}$ is holomorphic symplectic with respect to 
	\begin{equation}
	I_w=\frac{1-\bar{w}w}{1+\bar{w}w}I+\frac{i(w-\bar{w})}{1+\bar{w}w}J+\frac{w+\bar{w}}{1+\bar{w}w}K.
	\end{equation}
\end{proof}		
		
\subsubsection{Existence and uniqueness}

We next consider \eqref{CstarHitchinequation} from the perspective of moment maps: given a flat connection satisfying a suitable stability
condition, we wish to find a Hermitian metric solving \eqref{CstarHitchinequation}.
	
\begin{proposition}
Suppose that $(\ME,\vp)$ is a stable Higgs bundle and $H_0$ the corresponding harmonic metric; write the associated fields as $(A_{H_0},\phi_{H_0})$.
Define $\MP_1^{w}:=(D_{\bz}-w\phi_{\bz})d\bz$ and $\MP_2^w:=(\MD_z+w^{-1}\phi_z)dz$. Then $(\MP_1^w,\MP_2^w,H_0)$ solves
\eqref{CstarHitchinequation}. 

Conversely, let $(\MP_1^w,\MP_2^w,H)$ solve the $\CST$-Hitchin equations with parameter $w\in\CST$, and define
$$
\MD_{\bz}=\frac{1}{1+|w|^2}\MP_1^w+\frac{|w|^2}{1+|w|^2}(\MP_2^w)^{\da_H},\;\phi_z=\frac{w}{1+|w|^2}(\MP_2^w-(\MP_1^w)^{\da_H}).
$$
Then $(\MD_{\bz},\phi_z,H)$ solves the untwisted Hitchin equations.
\label{solutioncstarhitchin}
\end{proposition}
The proof follows by unwinding the definitions. 

\medskip

We next recall Corlette's theorem \cite{Corlette88}:
\begin{theorem}{\cite{Corlette88}}
\label{Thm_CorletteExistence}
Let $\nabla=\MP_1+\MP_2$ be an irreducible flat connection, with $(0,1)$ and $(1,0)$ parts $\MP_1$ and $\MP_2$, respectively.
Then there exists a unique Hermitian metric $H$ such that $\bar{\pa}_H\vp_H =0$, where 
\begin{equation}
\bar{\pa}_H:=\frac{1}{2}(\MP_1+\MP_2^{\da_H}),\;\vp_H=\frac{1}{2}(\MP_2-\MP_1^{\da_H}).
\label{Eq_corlette}
\end{equation} 
\end{theorem}

The fact that \eqref{CstarHitchinequation} is solvable generalizes this theorem; indeed, the equation in Corlette's theorem is the
$\CST$ Hitchin equation with coefficient $w=-i$.  

\begin{theorem}
\label{Thm_existenceofCstarHitchin}
Let $\nabla$ be an irreducible flat connection and fix any $w\in\CST$. Then there exists a solution to the $w$-twisted Hitchin equations. 
\end{theorem}
\begin{proof}
Write $\nabla=\MP_1+\MP_2$, and define, for any $H$, 
\begin{equation*}
\begin{split}
&\MD_{\bz}:=\frac{1}{1+|w|^2}\MP_1+\frac{|w|^2}{1+|w|^2}(\MP_2)^{\da_H},\;\phi_z:=\frac{w}{1+|w|^2}(\MP_2-(\MP_1)^{\da_H}),\\
& \MD_z:=\frac{1}{1+|w|^2}(\MP_1^{\da_H}+|w|^2\MP_2),\;\phi_{\bz}:=\frac{\bar{w}}{1+|w|^2}(\MP_2^{\da_H}-\MP_1)
\end{split}
\end{equation*}
and 
\begin{equation}
  \begin{split}
    &\MD':=\MD_z+\MD_{\bz}+i\phi_z+i\phi_{\bz} = \\
    & \frac{1}{(1+|w|^2)}((1-i\bar{w})\MP_1+(iw+|w|^2)\MP_2+(1-iw)\MP_1^{\da_H}+(|w|^2+i\bar{w})\MP_2^{\da_H}). 
\end{split}
\end{equation}

If $s$ is a section such that $\MD's=0$, then 
\begin{equation}
(1-i\bar{w})\MP_1s+(|w|^2+i\bar{w})\MP_2^{\da_H}s=0,\;(1-iw)\MP_1^{\da_H}s+(|w|^2+iw)\MP_2s=0.
\end{equation}
When $w\in\CST$, this implies $\MP_1s=\MP_2s=0$. As $\nabla$ is irreducible, we obtain that $\Ker\;\MD'=0$, hence $\MD'$ is irreducible
too. Applying Theorem \ref{Thm_CorletteExistence} to the operator $\MD'$, there exists a pair $(A_{H_0},\phi_{H_0})$ with Hermitian metric $H_0$ solves the Hitchin equation by Proposition \ref{solutioncstarhitchin}, the corresponding $(\MP_1^w,\MP_2^w,H_0)$ solves the w-twisted Hitchin equations.
\end{proof}	

\begin{proposition}{\cite{donaldson1985anti, Donaldson1987Infinite, Simpson1988Construction}}
	\label{Prop_uniqueness}
Let $\nabla=\MP_1^w+\MP_2^w$ be irreducible flat connection, then the solution to the w-twisted Hitchin equations is unique.
\end{proposition}
\begin{proof}
  For any Hermitian metric $H$, define $\Xi_H:=\Lambda([\MP_1^{w},(\MP_1^w)^{\da_H}]-w^2[\MP_2^{w},(\MP_2^{w})^{\da_H}])$.
  If $K$ is another Hermitian metric, write $H=Ke^s$ and consider $\Xi$ as a function of $s$.  Define the functional
\begin{equation}
\MM(H,K)=\int_0^1\int_{\Sigma}\langle s,\Xi_{us}\rangle _K\ \omega\we du, 
\end{equation}
where $\omega$ is the area form on $\Si$.  This functional reveals the variational structure for the $\CST$-Hitchin equations. Indeed, 
writing $H_t=Ke^{ts}$, then
\begin{equation}
\begin{split}
\frac{d}{dt}\MM(H_t,K)=\int_{\Sigma}\Tr(\Xi_{H_t}s)\omega,\;\frac{d^2}{dt^2}\MM(H_t,K)=\int_{\Si}\cos^2\beta|\MP_1^ws|^2+\sin^2\beta|\MP_2^ws|^2.
\end{split}
\end{equation}
When $\nabla$ is irreducible, $\frac{d^2}{dt^2}\MM(H_t,K)>0$, and this strict convexity implies uniqueness.
\end{proof}

\begin{subsection}{The boundary condition at infinity}
\label{Sec_boundaryconditioninfinity}
In the rest of the paper, we assume that $w=\tan\beta$ and write $I_{\beta}:=I_{\tan\beta}=
\cos2\beta I+\sin 2\beta K,$ and $\MM^{\beta}_{\Opers}:=\MM^{\tan\beta}_{\Opers}$ for
the corresponding oper moduli space.

The asymptotic boundary condition for \TEBE as $y\to\infty$ corresponds to the requirement that solutions converge to a $y$-independent
flat twisted connection, and that $\phi_1 \to 0$. 
More explicitly,  the boundary conditions for the triple $(A,\phi,\phi_1)$ are that $(A,\phi)$ converges
to a flat $w$-twisted $\SL(n,\mathbb{C})$ connection $\nabla^w$ associated to  the pair $(A_{\flat},\phi_{\flat})$, 
and that $\phi_1$ converges (exponentially) to $0$.

We can phrase this equivalently in terms of the Hermitian metric: 
\begin{definition}
Suppose that $A_{\flat}+i\phi_{\flat}$ is an irreducible flat connection, and define the (necessarily commuting) operators
$$
\MD_1:=(D_{(A_{\flat})_{\bz}}-(\phi_{\flat})_{\bz}\tan\beta)d\bz, \ \ \MD_2:=(D_{(A_{\flat})_z}+(\phi_{\flat})_z\cot \beta)dz,
$$
By Proposition \ref{solutioncstarhitchin}, we obtain a solution $H_\flat$; this is the boundary condition for the Hermitian metric at infinity.
\end{definition}

\end{subsection}
\end{section}
	
\begin{section}{The singular boundary Condition and the Model Metric}
\label{Sec_KobayashiHitchinMap}
\begin{subsection}{The tilted Nahm pole boundary conditions}
Fix a principal embedding $\slf_2 \hookrightarrow \slf_{n}$ and choose a basis $\slf_2=\mathrm{span}\, \{\mfe^+,\mfe^-,\mfe^0\}$ where
$$
[\mfe^+,\;\mfe^-]=\mfe^0,\;[\mfe^0,\;\mfe^-]=-2\mfe^-,\;[\mfe^0,\;\mfe^+]=2\mfe^+.
$$ 
Set $B_i=i(n-i)$ and write
\begin{equation}
\begin{split}
\mfe^+=\begin{pmatrix}
0 & \sqrt{B_1} & 0  &\cdots& 0\\
0 & 0 & \sqrt{B_2}  &\cdots& 0\\
\vdots & \vdots &\ddots &  &\vdots\\
\vdots & \vdots & &\ddots  &\sqrt{B_{n-1}}\\
0& 0 & \cdots & 0 &0
\end{pmatrix},\;\mfe^0=\begin{pmatrix}
n-1 & 0 & 0  &\cdots& 0\\
0 & n-3 &  0 &\cdots& 0\\
\vdots & \vdots &\ddots &  &\vdots\\
\vdots & \vdots & &-(n-3)  &0\\
0& 0 & \cdots  &0 &-(n-1)
\end{pmatrix}.
\end{split}
\end{equation}
	
Using local coordinate $(z,y)$ on $\Si\ti\RP$, we have $A=A_zdz+A_{\bz}d\bz,\;\phi=\phi_zdz+\phi_{\bz}d\bz$, and the
model tilted Nahm pole solution over $T^2\ti\RP$ is
\begin{equation}
\begin{split}
&A_z=y^{-1}\mfe^+\sin\beta,\;A_{\bz}=y^{-1}\mfe^{-}\sin\beta,\\
&\phi_z=y^{-1}\mfe^+\cos\beta,\;\phi_{\bz}=y^{-1}\mfe^{-}\cos\beta,\;\phi_1=\frac{i}{2y}\mfe^0\cos\beta.
\end{split}
\end{equation}
\begin{definition}[\cite{witten2011fivebranes}]
\label{Nahmpoleunitarydefinition}
The fields $(A,\phi,\phi_1)$ satisfy the tilted Nahm pole boundary conditions if in some local trivialization,
\begin{equation*}
\begin{split}
&A_z\sim y^{-1}\mfe^+\sin\beta+\MO(y^{-1+\ep}),\;A_{\bz}\sim y^{-1}\mfe^{-}\sin\beta+\MO(y^{-1+\ep}),\\
&\phi_z\sim y^{-1}\mfe^+\cos\beta+\MO(y^{-1+\ep}),\;\phi_{\bz}\sim y^{-1}\mfe^{-}\cos\beta+\MO(y^{-1+\ep}),\;\phi_1\sim \frac{i}{2y}\mfe^0\cos\beta+\MO(y^{-1+\ep}),
\end{split}
\end{equation*}
as $y\to 0.$
	\end{definition}
	
We now examine these boundary conditions in a parallel holomorphic gauge. In a local holomorphic coordinate $z$, we have
$\MD_1=D_{\bz}$ and $\MD_3=\pa_y$, and we also write $\MD_2=\pa+\al$. Suppose now that in some local trivialization,
\begin{equation}
\label{alformNahmpole}
\begin{split}
\al =\begin{pmatrix}
\st & \sqrt{B_1} & 0 &\cdots& 0\\
\st & \st & \sqrt{B_2} &\cdots& 0\\
\vdots & \vdots &\ddots & &\vdots\\
\vdots & \vdots & &\ddots &\sqrt{B_{n-1}}\\
\st& \st & \cdots  &\st &\st
\end{pmatrix}\, dz,
\end{split}
\end{equation}
where all entries two or more `levels' above the main diagonal are zero. Imposing the singular Hermitian metric
$H_0=\exp(-\log (y\cos\beta)\mfe^0)$ leads to fields satisfying the tilted Nahm pole boundary condition. This
follows directly from the expressions \eqref{unitarygaugeformula} in unitary gauge.
If $H=H_0e^s$ is another Hermitian metric, where $s$ is a section of $i\su(E.H_0)$ with $\sup|s|+y|ds|\leq Cy^{\ep}$, 
then the corresponding fields $(A_H,\phi_H,(\phi_1)_H)$ also satisfy the tilted Nahm pole boundary conditions.
\begin{definition}
If in some local trivialization, $\MD_2=\pa+\al$ with $\al$ as in \eqref{alformNahmpole}, and we set 
$H_0=\exp(-\log (y\sin \beta)\, \mfe^0)$, then any $H = H_0 e^s$ with $|s| + |y\, ds| < Cy^\ep$ satisfies the tilted Nahm pole boundary condition.
\end{definition}
\end{subsection}

\begin{subsection}{Holomorphic data from the singular boundary condition}
By \eqref{complexgaugeactionchange}, up to complex gauge transform, a choice of holomorphic data corresponds to a flat connection over $\Si$.
We now discuss how these singular boundary conditions interact with this holomorphic data. 
	
As above, in some trivialization near $y=0$, these fields satisfy
\begin{equation}
\begin{split}
\MD_1=&\pa_{\bz}+A_{\bz}-\phi_{\bz}\tan\beta +\MO(y^{-1+\ep})=\pa_{\bz}+\MO(y^{-1+\ep}), \\
\MD_2=&\pa_z+\frac{\mfe^+}{y\sin \beta}+\MO(y^{-1+\ep}) \\
\MD_3=&\pa_y+\frac{\mfe^0}{2y}+\MO(y^{-1+\ep})
\end{split}
\label{NPdatatrivilasection}
\end{equation}	
Choose a holomorphic basis $s_1,\cdots,s_{n+1}$ so $s_i$ corresponds to $(0,\cdots,0,1,0,\cdots,0)^{\da}$.
If $s(y)$ is any section of $E$, we can solve the ODE $\MD_3s=0$ with initial value $s(y)|_{y=1}=\sum_{i=1}^{n}a_is_i$, $a_i\in\mathbb{C}$.
Then 
$$
s(y)=\sum_{i=1}^{n}a_iy^{-\frac{n-1}{2}+i-1} s_i +\MO(y^{-\frac{n-1}{2}+i-1+\ep})   \ \mbox{as}\ y\searrow  0,
$$
for some small $\ep > 0$. This determines a filtration $E_{\bullet}: 0=E_0\subset E_1\subset \cdots \subset E_{n}=E$ by vanishing rates, where 
\begin{equation}
\begin{split}
E_i:=\{s\in\Gamma(E)|\ \MD_3s=0,\;\lim_{y\to0}|sy^{-\frac{n-1}{2}+i-1+\ep}|=0\}
\end{split}
\end{equation}
for any $0<\ep<1$.  (Note that, by construction, $\rank\;E_i=i$.) 

To see that the boundary condition induces an oper structure, we must check the three points in Definition \ref{Def_Opers}.
The first two are straightforward. The third follows from the 
\begin{proposition}
For any section $s$ of $E_i$, both $\MD_1s$ and $\MD_2 s$ lie in $E_{i+1}\otimes \bar{K}$. 
\end{proposition}
\begin{proof}
If $s$ is a smooth section of $E_i$ with $\MD_3 s = 0$, then $0 = [\MD_j,\MD_3]s=\MD_3\MD_js$, $j=1,2$.
But then, integrating this singular ODE in the $y$ direction shows that 
$\lim_{y\to0}|(\MD_js)y^{-\frac{n-1}{2}+i+\ep}|=0$, $j = 1, 2$. 
\end{proof}
To conclude the verification that this data gives an oper, observe finally that since the entries of $\MD_j$, $j = 1,2$, above
the main diagonal are nonzero constants, each
$$
\MD_2:E_{i}/E_{i-1}\to E_{i+1}/E_{i}\otimes K
$$
is an isomorphism.
	
In summary, we have shown that the tilted Nahm pole boundary conditions determine an oper structure on solutions.
\end{subsection}

\begin{subsection}{The moduli space and the Gaiotto-Witten map}
We can now define the moduli space of solutions to the twisted Bogomolny equations with tilted Nahm pole boundary conditions
\begin{multline*}
  \MM_{\mathrm{TBE}}^{\beta}:=\{(A,\phi,\phi_1)| \ \mathrm{TBE}(A,\phi,\phi_1)=0, \ (A,\phi,\phi_1) \ \mbox{satisfies the tilted Nahm Pole } \\
\quad \mbox{boundary condition at} \ y=0,\ \mbox{and} \ (A,\phi,\phi_1) \mbox{ converges} \\
\mbox{to a flat $\SL(n,\CC)$ connection as} \ y\to+\infty\}/\mathcal{G}_0;
\end{multline*}
here $\mathcal{G}_0$ is the (real) gauge group preserving the Nahm pole boundary condition and decay assumptions.

\begin{proposition}{\cite{gaiotto2012knot}} \label{INP}
  There is a well-defined Kobayashi-Hitchin map
  \[
    I_{\Opers}^{\beta}:\MM_{\mathrm{TBE}}^{\beta}\to \MM_{\mathrm{Oper}}^{\beta}.
    \]
\end{proposition}	
\end{subsection}
\end{section}
	
\begin{section}{Linear Analysis}
We now study Fredholm properties of the linearized moment map. 

Fix a background Hermitian metric $H_0$ on the bundle $E$ and the three operators $\MD_i$ satisfying $[\MD_i,\MD_j]=0$.
We seek a new Hermitian metric $H=H_0e^s$ so that the final moment map equation
\begin{equation}
\Omega_H:=\frac{i}{2}\Lambda(\cos^2\beta[\MD_1, \MD_1^{\da}]-\sin^2\beta[\MD_2,\MD_2^{\da}])+[\MD_3,\MD_3^{\da}] = 0
\label{generalizaedcurvature}
\end{equation}
holds; here $\Lambda$ is contraction by the K\"ahler form. 

\begin{definition} A Hermitian metric $H_0$ is called admissible if:
\begin{itemize}
\item the Chern connection associated to $H_0$ satisfies the Nahm pole boundary conditions;
\item the Chern connection converges to an oper as $y \to \infty$, cf.\ Section \ref{Sec_boundaryconditioninfinity};
\item $\Omega_{H_0}$ vanishes to all orders at $y=0$.
\end{itemize}
\end{definition}

The following adapts a result in \cite{HeMazzeo2018} to the present setting:
\begin{proposition}  Defining $H = H_0 e^s$, then
\begin{equation}
\Omega_H=\Omega_{H_0}+\gamma(-s)\ML_{H_0}s+Q(s),
\label{quasilinearformequation}
\end{equation}
where 
\[
 \ML_{H_0}s:=\frac{i}{2}\Lambda(\cos^2\beta\MD_1\MD_1^{\da_{H_0}}-\sin^2\beta\MD_2\MD_2^{\da_{H_0}}) s
 +\MD_3\MD_3^{\da_{H_0}}s,
\]
and 
$$
Q(s):=\frac{i}{2}\Lambda(\cos^2\beta\MD_1(\gamma(-s))\MD_1^{\da_{H_0}}s-\sin^2\beta\MD_2(\gamma(-s))
\MD_2^{\da_{H_0}}s)+\MD_3\gamma(-s)\MD_3^{\da_{H_0}}s,
$$
where $\gamma(s):=\frac{e^{\mathrm{ad}_s}-1}{\mathrm{ad}_s}$. Furthermore, 
\begin{equation}
\begin{split}
\label{Keyequation2}
\langle \Omega_{H}-\Omega_{H_0},s\rangle _{H_0}=\Delta|s|_{H_0}^2+|v(s)\na s|^2
\end{split}
\end{equation}
where $v(s)=\sqrt{\gamma(-s)}=\sqrt{\frac{1-e^{-ad_s}}{ad_s}}$, $\Delta=\Delta_{\Sigma}-\pa_y^2$
and 
\begin{equation}
\begin{split}
  |v(s)\na s|^2:=&\frac{1}{4}\cos^2\beta(|v(s)\MD_1s|^2+|v(s)\MD_1^{\da_{H_0}}|^2)+\frac{1}{4}\sin^2\beta(|v(s)\MD_2s|^2
  +|v(s)\MD_2^{\da_{H_0}}|^2)\\&+\frac{1}{2}(|v(s)\MD_3s|^2+|v(s)\MD_3^{\da_{H_0}}|^2).
\end{split}
\label{vsdefinition}
\end{equation}
\end{proposition}	
\begin{proof}
First we have $\MD_i^{\da_H}=\MD_i^{\da_{H_0}}+e^{-s}(\MD_i^{\da_{H_0}}e^s)$. Next, if $X(w)$ is any smooth family of
Hermitian metrics, then
\begin{equation}
\pa_w e^{X}=e^{X}\gamma(-X)\pa_w X =\gamma(X)\pa_w X e^{X},
\label{Exponentialderivative}
\end{equation}
where $\del_w$ denotes a `generic' derivative, so e.g.\ could be $\MD_i$ or $\MD_i^{\dag_H}$.  Using these, 
we now compute
\begin{equation*}
\begin{split}
 \Omega_H&=\Omega_{H_0}+\frac{i}{2}\Lambda(\cos^2\beta\MD_1(e^{-s}\MD_1^{\da_{H_0}}e^s)-\sin^2\beta
 \MD_2(e^{-s}\MD_2^{\da_{H_0}}e^s))+ \MD_3(e^{-s}\MD_3^{\da_{H_0}}e^s)\\
 &=\Omega_{H_0}+\frac{i}{2}\Lambda(\cos^2\beta\MD_1(\gamma(-s)\MD_1^{\da_{H_0}}s)-\sin^2\beta\MD_2(\gamma(-s)
 \MD_2^{\da_{H_0}}s))+ \MD_3(\gamma(-s)\MD_3^{\da_{H_0}}s)\\ &=\Omega_{H_0}+\gamma(-s)\ML_{H_0}s+Q(s).
\end{split}
\end{equation*}
The identity \eqref{Keyequation2} requires an analysis of the expression
\begin{equation*}
\begin{split}
&\langle \Omega_H-\Omega_{H_0},s\rangle _{H_0}\\
=&\langle \frac{i}{2}\Lambda(\cos^2\beta\MD_1(\gamma(-s)\MD_1^{\da,H_0}s)-\sin^2\beta
\MD_2(\gamma(-s)\MD_2^{\da,H_0}s))+\MD_3(\gamma(-s)\MD_3^{\da,H_0}s),s\rangle _{H_0}.
\end{split}
\end{equation*}
We consider the summands in turn.
	
For the first, we compute 
\begin{equation*}
\begin{split}
\lan\frac{i}{2}\Lambda\MD_1(\gamma(-s)\MD_1^{\da_{H_0}}s),s\ran&=\frac{i}{2}\Lam\bar{\pa}\lan\gamma(-s)\MD_1^{\da_{H_0}}s,s 
\ran+\frac 12 \lan \gamma(-s)\MD_1^{\da_{H_0}}s,\MD_1^{\da_{H_0}}s\ran\\
&=\frac{i}{2}\Lambda \bar{\pa}\lan\MD_1^{\da_{H_0}}s,\gamma(-s)s \ran+\frac{1}{2}|v(-s)\MD_1^{\da_{H_0}}s|^2\\
&=i\Lambda \bar{\pa}\pa|s|^2+\frac{1}{2}|v(-s)\MD_1^{\da_{H_0}}s|^2\\
&=\frac{1}{2}\Delta_{\Sigma}|s|^2+\frac{1}{4}|v(-s)\MD_1^{\da_{H_0}}s|^2+\frac{1}{4}|v(-s)\MD_1s|^2.
\end{split}
\end{equation*}
	
The second term follows from
\begin{equation*}
\begin{split}
\lan\frac{i}{2}\Lambda\MD_2(\gamma(-s)\MD_2^{\da_{H_0}}s),s\ran&=\frac i2\Lam{\pa}\lan\gamma(-s)\MD_2^{\da_{H_0}}s,s 
\ran-\frac 12\lan \gamma(-s)\MD_2^{\da_{H_0}}s,\MD_2^{\da_{H_0}}s\ran\\
&=\frac i2\Lambda \pa\lan\MD_2^{\da_{H_0}}s,\gamma(-s)s \ran-\frac 12|v(-s)\MD_2^{\da_{H_0}}s|^2\\
&=i\Lambda \pa\bar{\pa}|s|^2-\frac12 |v(-s)\MD_2^{\da_{H_0}}s|^2\\
&=-\frac12\Delta_{\Sigma}|s|^2-\frac{1}{4}|v(-s)\MD_2^{\da_{H_0}}s|^2-\frac{1}{4}|v(-s)\MD_2s|^2.
\end{split}
\end{equation*}
	
Finally, calculating as previously, the third term equals
\begin{equation*}
\begin{split}
\lan \MD_3(\gamma(-s)\MD_3^{\da_{H_0}}s),s\ran&=\pa_y\lan \gamma(-s)\MD_3^{\da_{H_0}}s,s\ran+ 
\lan \gamma(-s)\MD_3^{\da_{H_0}}s,\MD_3^{\da_{H_0}}s\ran\\ &=-\pa_y^2|s|^2+\frac{1}{2}(|v(s)\MD_3s|^2+|v(s)\MD_3^{\da_{H_0}}|^2).
\end{split}
\end{equation*}
	
Combining these yields the desired identity. 
\end{proof}

Next recall the K\"ahler identities \cite[Lemma 3.1]{Simpson1988Construction}: 
\begin{equation}
\begin{split}
&i[\Lambda,\MD_1]=(\MD_1^{\da_H})^{\st},\ \ i[\Lambda,\MD_1^{\da_H}]=-(\MD_1)^{\st}\\
&i[\Lambda,\MD_2]=-(\MD_2^{\da_H})^{\st},\ \ i[\Lambda,\MD_2^{\da_H}]=(\MD_2)^{\st}
\end{split}
\end{equation}
Noting also that $(\MD_3^{\da_{H_0}})^{\st}=\MD_3$, we conclude
\begin{corollary}
$\ML_{H_0}=\frac{1}{2}(\cos^2\beta(\MD_1^{\da_{H_0}})^{\st}\MD_1^{\da_{H_0}}+\sin^2\beta(\MD_2^{\da_{H_0}})^{\st}\MD_2^{\da_{H_0}})+
\MD_3 \MD_3^{\da_{H_0}}$.  
\end{corollary}

\begin{proposition}
Setting $\na_1=\MD_1+\MD_1^{\da}, \na_2=\MD_2+\MD_2^{\da}$, then there are Weitzenb\"ock formulae
\begin{equation}
(\MD_1^{\da})^{\st}\MD_1^{\da}=\frac{1}{2}\na_1^{\st}\na_1+\frac{i}{2}\Lambda[\MD_1, \MD_1^{\da}], \qquad \;(\MD_1)^\st\MD_1
=\frac{1}{2}\na_1^{\st}\na_1-\frac{1}{2}\Lambda[\MD_1,\MD_1^{\da}],
\label{W1}
\end{equation} 
\begin{equation}
(\MD_2^{\da})^{\st}\MD_2^{\da}=\frac{1}{2}\na_2^{\st}\na_2+\frac{i}{2}\Lambda[\MD_2,\MD_2^{\da}],\qquad \;(\MD_2)^\st\MD_2
=\frac{1}{2}\na_2^{\st}\na_2-\frac{1}{2}\Lambda[\MD_2,\MD_2^{\da}],
\label{W2}
\end{equation}
and also 
\begin{equation}
  \MD_3\MD_3^{\da}=-(\MD_y^2+\frac{\phi_1^2}{\cos^2\beta})+\frac{1}{2}[\MD_3,\MD_3^{\da}],\;\MD_3^{\da}\MD_3=-(\MD_y^2+
  \frac{\phi_1^2}{\cos^2\beta})- \frac{1}{2}[\MD_3,\MD_3^{\da}].
\label{W3}
\end{equation}
\end{proposition}
\begin{proof}
We compute
\begin{equation}
\begin{split}
\na_1^{\st}\na_1&=\MD_1^{\st}\MD_1+(\MD_1^{\da})^{\st}\MD_1^{\da}=-i\Lambda \MD_1^{\da}\MD_1+i\Lambda \MD_1\MD_1^{\da} \\
& =2i\Lambda \MD_1\MD_1^{\da}-i\Lambda [\MD_1,\MD_1^{\da}]  
=-2i\Lam \MD_1^{\da}\MD_1+i\Lambda [\MD_1,\MD_1^{\da}].
\end{split}
\end{equation}
	
For $\MD_2$, we have
\begin{equation}
\begin{split}
\na_2^{\st}\na_2&=\MD_2^{\st}\MD_2+(\MD_2^{\da})^{\st}\MD_2^{\da}=-i\Lambda \MD_2^{\da}\MD_2+i\Lambda \MD_2\MD_2^{\da} \\
& =-2i\Lambda \MD_2\MD_2^{\da}+i\Lambda [\MD_2,\MD_2^{\da}]  
=2i\Lam \MD_2^{\da}\MD_2-i\Lambda [\MD_2,\MD_2^{\da}].
\end{split}
\end{equation}
	
Finally, to obtain \eqref{W3}, the formulas $\MD_3=\MD_y-i\frac{\phi_1}{\cos\beta}$ and $\MD_3^{\da}=-\MD_y-i\frac{\phi_1}{\cos\beta}$
lead to $(\MD_3-\MD_3^{\da})^2=4\MD_y^2,$ $(\MD_3)^2+(\MD_3^{\da})^2=2\MD_y^2-2\frac{\phi_1^2}{\cos^2\beta}$.
In addition, 
$$
(\MD_3-\MD_3^{\da})^2=\MD_3^2+(\MD_3^{\da})^2-\MD_3\MD_3^{\da}-\MD_3^{\da}\MD_3,
$$
so altogether
\begin{equation*}
\begin{split}
2(\MD_y^2+\frac{\phi_1^2}{\cos^2\beta})=-\MD_3\MD_3^{\da}-\MD_3^{\da}\MD_3=-[\MD_3,\MD_3^{\da}]-2\MD_3^{\da}\MD_3
=[\MD_3,\MD_3^{\da}]-2\MD_3\MD_3^{\da}.
\end{split}
\end{equation*}
\end{proof}
This leads to a simpler expression for $\ML_H$: 
\begin{corollary}
\begin{equation}
 \ML_H=\frac{1}{4}(\cos^2\beta\, \na_1^{\st}\na_1+\sin^2\beta\, \na_2^{\st}\na_2)-(\MD_y^2+
 \frac{\phi_1^2}{\cos^2\beta})+\frac{1}{2}[\Omega_H,\cdot ],
 \label{MLH}
\end{equation}
where $\phi_1^2=[\phi_1,[\phi_1,\;]]$.

\end{corollary}

\begin{subsection}{Indicial Roots}
  The mapping properties and regularity of solutions for the operator $\ML_H$ in \eqref{MLH} rely on the determination
  of the indicial roots of this operator. Note that the final term in $\ML_H$ which involves $\Omega_H$ is absent since
  we work at an exactx solution. 

  First some notation. Recall the component operators
  $$
  \MD_1=(D_{\bz}-\phi_{\bz}\tan\beta)d\bz,\; 	\MD_2=(D_z+\phi_z\cot \beta)dz,\; \MD_3=D_y-i\frac{\phi_1}{\cos \beta},
  $$
 and the model tilted Nahm singularities near $y=0$
\begin{equation*}
\begin{split}
&A_z\sim y^{-1}\mfe^+\sin\beta+\MO(y^{-1+\ep}),\;A_{\bz}\sim y^{-1}\mfe^{-}\sin\beta+\MO(y^{-1+\ep}),\\
&\phi_z\sim y^{-1}\mfe^+\cos\beta+\MO(y^{-1+\ep}),\;\phi_{\bz}\sim y^{-1}\mfe^{-}\cos\beta+\MO(y^{-1+\ep}),\;\phi_1\sim
\frac{i}{2y}\mfe^0\cos\beta+\MO(y^{-1+\ep}). 
\end{split}
\end{equation*}
Putting these together, we have 
\begin{equation}
\begin{split}
  &\cos\beta \, \MD_1=\cos\beta \, D_{A_{\bar{z}}}-\sin\beta\, \phi_{\bz}\sim \MO(1),\;\sin\beta \, \MD_2=\sin\beta \, D_{A_z}-
  \cos\beta \, \phi_z\sim  \frac{\mfe^+}{y},\\ &\MD_3=\MD_y-\frac{i\phi_1}{\cos\beta}\sim\pa_y+\frac{1}{2y}\mfe^0
\end{split}
\end{equation}

The operator $\ML_H$ is singular at $y=0$, and its behavior there is well-approximated by the so-called normal operator,
cf.\ \cite{Mazzeo1991}:
$$
N(\ML_{H_0}):=\Delta_{\mathbb{R}^3}-\frac{1}{2}([\mfe^+,[\mfe^-,s]+[\mfe^-,[\mfe^+,s]])-\frac{1}{4}[\mfe^0,[\mfe^0,s]].
$$
This is both translation-invariant in $z$ and dilation-invariant in $(z,y)$ jointly, and is identified with the linearization of the
nonlinear equations at the global model solution on $\RR^3_+$.

A key invariant of $\ML_H$, or equivalently of its normal operator $\ML_{H_0}$, is its associated set
of indicial roots. These are the formal rates of growth or decay of solutions at $y=0$. By definition, $\lambda$
is an indicial root of $\ML_H$ at $(z_0, 0)$ if there exists a section $s$ defined in a neighborhood of this point such that
$$
\ML_{H_0}(y^{\lam}s_0)= 0, \quad \mbox{where}\ s_0=s|_{z=z_0, y=0}.
$$
Using that $\ML_H (y^\lambda s) = N(\ML_{H_0})(y^{\lam}s_0)+\MO(y^{\lam-1})$, we see that $\lambda$ is an indicial root
if there exists some $s$ such that $\ML_H(y^{\lam}s)=\MO(y^{\lam-1})$, in contrast to the expected rate $\MO(y^{\lam-2})$.
	
From the explicit form of this operator, we see that $\lam$ is an indicial root if
\begin{equation}
\lam(\lam-1)s = \frac{1}{2}([\mfe^+,[\mfe^-,s]+[\mfe^-,[\mfe^+,s]])+\frac{1}{4}[\mfe^0,[\mfe^0,s]].
\end{equation}
The operator on the right is the Casimir operator for $\slf_2$,
$$
\Delta_{\mathrm{Cas}}s:=\frac{1}{2}([\mfe^+,[\mfe^-,s]+[\mfe^-,[\mfe^+,s]])+\frac{1}{4}[\mfe^0,[\mfe^0,s]],
$$
so $\lam$ is an indicial root for $\ML$ if and only if $\lam(\lam-1)$ is an eigenvalue for $\Delta_{\mathrm{Cas}}$.

Using the known values of this Casimir spectrum for $\mathfrak{sl}_{n}$, we arrive at a result mirroring the result and calculations
in \cite{MazzeoWitten2013}:
\begin{proposition}
The set of indicial roots of the twisted \EBE with tilted Nahm pole boundary conditions is
$\{-(n-1), \dots, -1, 2,\ldots, n\}$. 
\end{proposition}
\end{subsection}

\begin{subsection}{Function spaces and Mapping Properties}
Using the results of the last subsection and invoking the theory in \cite{Mazzeo1991}, we now state the
Fredholm theory for the operator $\ML_H$ acting on a family of weighted Holder spaces adapted to the degeneracy of the operator.
	
\begin{definition}
Define $\calC^{k,\alpha}_{\ie}( \Si \ti \RP)$ to be the space of all functions $u$ on $\Si_x \ti \RP_y$ such that
\begin{itemize}
\item[i)] In the region $\{y \leq 1\}$, 
\[
||u||_{L^\infty} + \sup_{i + |\beta| \leq k} [ (y\del_y)^i (y\del_x)^\beta u ]_{\ie; 0,\alpha} < \infty,
\]
where 
\[
[v]_{\ie; 0, \alpha} := \sup_{(y,x) \neq (y', x')\atop 
y, y' \leq 1}  \frac{|u(y,x) - u(y', x')|(y+y')^\alpha}{|y-y'|^\alpha + |x-x'|^\alpha}.
\]
\item[ii)] Away from all boundaries we require that $u$ lies in the ordinary H\"older space $\calC^{k,\alpha}$ 
on each slab $\Si \ti [L, L+1]$, uniformly for $L \geq 1$. 
\end{itemize}
\end{definition}
Fixing $0 < \al < 1$, for any $\mu,\delta \in \RR$, define
\begin{equation}
\begin{split}
\calX^k_{\mu, \delta}(\Si \ti \RP; \isu(E, H_0))=y^\mu e^{\delta y} \calC^{k,\alpha}_{\ie}(\Si \ti \RP) & = \{u = y^\mu e^{y\delta}v: v \in \calC^{k,\alpha}_{\ie} \}.
\end{split}
\end{equation}
\begin{theorem}{\cite{MazzeoWitten2013, HeMazzeo2018}}
Suppose $\mu\in(-1,2)$ and $\delta>0$; then for any $k\geq 0$ and $0<\al<1$, 
\begin{equation}
\ML_H:  \calX^{k+2}_{\mu, -\delta}(\Si \ti \RP; \isu(E, H_0)) \longrightarrow \calX^k_{\mu-2,  -\delta}(\Si \ti \RP; \isu(E, H_0)).
\label{Fred1map}
\end{equation}
is a Fredholm operator.
\end{theorem}
\begin{proof}
  The operator $\ML_H$ is the same as the operator considered in \cite[Theorem 5.8]{HeMazzeo2018} up to a compact operator, so
this result follows directly from that one.
\end{proof}
\end{subsection}

\end{section}
	
\begin{section}{Construction of approximate solutions}
Given an oper $(E,F_{\bullet},\na)$, we now construct an admissible Hermitian metric satisfying the tilted Nahm pole boundary conditions.

\begin{proposition}
  For every element $(E,F_{\bullet},\na)\in\MM_{\Op}$, there exists an Hermitian metric $H_0$ satisfying Nahm pole boundary conditions
  such that in unitary gauge relative to $H_0$, $\Omega_{H_0} = \MO(y^\infty)$.
\end{proposition}
\begin{proof}
By Theorem \ref{ClassificationofOpers}, given an oper $\na^{\beta}_{q}\in\MM_{\Op}$, we can write
  $$
  \nabla_{\mathbf{q}}^{\beta}:=\bar{\pa}+\pa^{\da_{h_0}}+\vpq+\vpz^{\da_{h_0}},
  $$ 
with respect to the trivialization $\ME=K^{-\frac{n-1}{2}}\oplus K^{-\frac{n}{2}+1}\oplus\cdots\oplus K^{\frac{n-1}{2}}$.
	
Now define $\MD_{1,\cpx}:=(\na^{\beta}_{q})^{0,1}=\bar{\pa}-\vpz^{\da_{h_0}},\;\MD_{2,\cpx}:=(\na^{\beta}_{q})^{1,0}=\pa^{\da_{h_0}}+\vpq$,
where the label "cpx" means we are working in a complex gauge.
	
Consider $H_0^{(0)}=\exp(-\log (y\sin \beta)\, \mfe_0)$ and $g=\sqrt{H_0^{(0)}}=\diag(\lam_1,\lam_2,\cdots,\lam_{n})$.
By the definition of $\mfe_0$, $\lam_i= {(y\sin\be)}^{-\frac{N}{2}+i-1}$ as an element of $\End(K^{-\frac{N}{2}+i-1},K^{-\frac{N}{2}+i-1})$.
Now set $\MD_{i,\app}:=g\MD_{i,\cpx} g^{-1}$ (where "app" indicates that this is an approximate solution).
We compute these operator explicitly.
	
We have $\MD_{1,\app}:=g\MD_{1,\cpx} g^{-1}=g^{-1}\bar{\pa}g+g^{-1}\vpz^{\da_{h_0}}g$. Since $g$ is a constant section,
$g^{-1}\circ \bar{\pa} \circ g=\bar{\pa}$.  By \cite{collier2014asymptotics}, $h_0$ is diagonal, so if $(\vpz^{\da_{h_0}})_{ij}$ denotes the $(i,j)$-component
of $\vpz^{\da_{h_0}}$, then $(\vpz^{\da_{h_0}})_{ij}=0$ for $j\neq i-1$. As $(g\vpz^{\da_{h_0}} g^{-1})_{ij}=\lam_i(\vpz^{\da_{h_0}})_{ij}\lam_j^{-1}$,
we have $g\vpz^{\da_{h_0}} g^{-1}\sim \MO(y)$.

Next, we compute $\MD_{2,\app}=g\MD_{2,\cpx} g^{-1}=\pa^{\da_{h_0}}+g\vpq g^{-1}$.  As before,
$g \vpq g^{-1}=(\lam_i \lam_j^{-1}\vp_{ij})$, where $\vpq=(\vp_{ij})$. We can decompose this as
$\MD_2=\pa_z+\phi_z^{\mathrm{mod}}+b$, where 
\begin{multline}
\phi_z^{\mathrm{mod}}={(y\sin\beta)}^{-1}
\begin{pmatrix}
0 & \sqrt{B_1} & 0 &\cdots& 0\\
0 & 0 & \sqrt{B_2} &\cdots& 0\\
\vdots &  &\ddots  & &\vdots\\
0 &  & &\ddots  &\sqrt{B_{n-1}}\\
0& 0 & \cdots & \cdots &0
\end{pmatrix} \mbox{and}\ \\
b=\begin{pmatrix}
0 & 0  & 0 &\cdots& 0\\
0 & 0 & 0 &\cdots& 0\\
\vdots &  &\ddots & &\vdots\\
0 &  & &\ddots  & 0\\
y^{n-1}q_{n}& y^{n-2}q_{n-1} & \cdots & yq_2 &0
\end{pmatrix}.\qquad 
\end{multline}
Since $b = \MO(y)$, then in the gauge defined by $g$, 
\[
\Omega_{H_0^{(0)}}=[\phi_z^{\mo},b^{\da}]+[b, (\phi_{z}^{\mo})^\da] =  \MO(1);
\]
even more specifically, the right hand side can be written $F_{H_0^{(0)}} + \MO(y)$. Clearly $\Omega_{H_0^{(0)}}$
depends continuously on $\mathbf{q}$.
	
We now add correction terms to make this error vanish to higher and higher order.  Indeed,
suppose that we have found a Hermitian metric $H_0^{(j)}$ such that $\Omega_{H_0^{(j)}} = F_jy^j + \MO(y^{j+1})$ for
some $j \geq 0$ (so $F_0 = F_{H_0^{(0)}}$ above), 
and define $H_0^{(j+1)} = H_0^{(j)} e^s$.  Using \eqref{quasilinearformequation}, we see that in order 
to show that $\Omega_{H_0^{(j+1)}} = F_{j+1}y^{j+1} + \MO(y^{j+2})$, it suffices to solve the equation 
\[
\gamma(-s)\ML_{H_0^{(j)}} s = -F_j y^j\ \ \mbox{modulo terms of order}\ y^{j+1 - \ep},
\]
But $\gamma(-s_j y^j) = \mathrm{Id} + \MO(y)$, and $\ML_{H_0^{(j)}}$ equals 
the normal operator $N(\ML_{H_0})$ to leading order, so we must solve
$N(\ML_{H_0}) s_j y^j = -F_j y^j$, where $s_j$ is an element of $i\su(E, H_0)$.  
This linear algebraic equation is solvable at least when $j$ is not an indicial root, and the solution
depends continuously on $\mathbf{q}$; in the exceptional cases where $j$ is an indicial root, 
one must replace $s_j y^j$ by $\tilde{s}_j y^j \log y$ to obtain a solution. 
(The possibility of these extra log factors is why the error has been written as $\MO(y^{j+1-\ep})$.)
In any case, we can carry out this inductive procedure and then take a Borel sum to obtain a Hermitian endomorphism 
\[
s \sim \sum_{j =0}^\infty s_{j\ell} \, y^j (\log y)^\ell
\]
(with $s_{0\ell} = 0$ for $\ell > 0$, and with only finitely many $s_{j \ell}$ nonzero for each $j$), 
such that if we set $H_0 = H_0^{(0)} e^s$, then $\Omega_{H_0} = \MO(y^N)$ for every $N \geq 0$.
\end{proof}

In summary, we obtain the
\begin{theorem}
\label{existadmissble1}
For any $(E,F_{\bullet},\na_{\mathbf{q}})\in\MM_{\Opers}^{\beta}$, there exists an admissible Hermitian metric $H_0$;
the approximate solution $\Omega_{H_0}(\mathbf{q})$ depends continuously on $\mathbf{q}$
\end{theorem}
\end{section}
	
\begin{section}{Continuity Method}
We now solve the extended Bogomolny equations with these boundary conditions using the standard method of continuity; this argument is close
to the one in \cite{HeMazzeo2018}, so our treatment is brief.
\begin{subsection}{Method of continuity}
Given an oper $(E,F_{\bullet},\na)\in\MM_{\Opers}^{\beta}$, fix an admissible approximate solution $H_0$ to the equation $\Omega_H=0$.
  	
Let $i\mathfrak{su}(E,H_0)$ be the subspace of Hermitian endormorphisms in $\End(E)$ preserving $H_0$. For any $s\in i\mathfrak{su}(E,H_0)$,
define the new Hermitian metric $H=H_0e^{s}$ and the family of maps
\begin{equation}
N_t(s):=\Ad(e^{\frac s2})\Omega_{H}+ts=0.
\label{eqmethodofContinuation}
\end{equation}
Note that $\Omega_H\in i\mathfrak{su}(E,H)$ and $\Ad(e^{\frac{s}{2}}):\isu(E,H)\to\isu(E,H_0)$ is a bundle isomorphism satisfying
\[
\langle \Ad(e^{\frac{s}{2}})f, \Ad(e^{\frac{s}{2}})g\rangle _{H}=\langle f,g\rangle _{H_0} \ \mbox{where}\ 
f,g \ \mbox{are sections of}\  i\su(E,H). 
\]
Define
\begin{equation}
I:=\{t\in[0,1]:N_t(s)=0 \ \ \mbox{has a solution} \ \ s\in \calX^{k+2}_{2-\ep, -\delta} \};
\label{setI}
\end{equation}
we show that $I$ is nonempty, open and closed, so that $I=[0,1]$, and the problem is solved.
\end{subsection}

\begin{subsection}{$I$ is nonempty}
\begin{proposition}
There exists an admissible Hermitian metric $H_0$ and section $s$ such that $N_1(s)=0$.
\end{proposition}
\begin{proof}
  Following \cite{martin1995kobayashi}, pick up any admissible Hermitian metric $H_{-1}$, write $\kappa=\Omega_{H_{-1}}$ and 
  define $H_0=H_{-1}e^{\kappa}$. Then if we set $s=-\kappa$, we have that $N_1(-\kappa)
  = \Ad(e^{\frac{-\kappa}{2}})\Omega_{H_0e^{-\kappa}}-\kappa=\Omega_{H_{-1}}-\kappa=0$.
\end{proof}
\end{subsection}

\begin{subsection}{Openness}
  We next study the linearization more closely. Assume $s$ satisfies $N_t(s)=0$ for some $t$ and define 
  $$
  \ML_{t,s}(s'):=\frac{d}{du}|_{u=0}N_t(s+us').
  $$
Using the computations in \cite{HeMazzeo2018}, we have the
\begin{proposition}{\cite[Proposition 6.2, 6.4]{HeMazzeo2018}}
  \label{linearizationproperties}
Suppose that $N_t(s) = 0$. Then for any sections $s_1, s_2$ of $\calX^{k+2}_{\mu, -\delta}(\Si \ti \RP; \isu(E, H_0))$, we have
  $$
  \begin{array}{rll}
& i)  \qquad & \ML_{t,s}(s_1) =Ad(e^{\frac{s}{2}})\ML_Hs_1+ts_1; \\[0.5ex]
& ii)  & \int\langle \ML_{t,s} s_1, \Ad(e^{\frac{s}{2}})s_1\rangle _{H_0} =\int\sum_{i=1}^3|\MD_i^{\da_H}s_1|^2_{H}+t|\Ad(e^{\frac s4})s_1|^2_{H_0}; \\[0.5ex]
& iii)  & \int \lan \ML_{t,s}s_1, s_2\ran_{H_0} =\int\lan s_1,Ad(e^{\frac{s}{2}})\ML_{t,s}(Ad(e^{-\frac{s}{2}})s_2)\ran.
\end{array}
$$
(Note that part i) is purely algebraic and does not require any assumptions on the decay of $s_1$.)
\end{proposition}
	
\begin{proposition}
$\ML_{t,s}:  \calX^{k+2}_{\mu, -\delta}(\Si \ti \RP; \isu(E, H_0)) \longrightarrow \calX^k_{\mu-2,  -\delta}(\Si \ti \RP; \isu(E, H_0))$
is an isomorphism.
\end{proposition} 
\begin{proof}
If $s\in \Ker\;\ML_{t,s}$, then by ii) of this last Proposition, $\Ad(e^{\frac{s}{4}})s=0$ and hence $s=0$. Part iii) shows that
the range of $\ML_{t,s}$ is dense. Since this operator is Fredholm, its range is closed, and hence it is an isomorphism.
\end{proof}
	
\begin{proposition}
\label{Iopen}
$I$ is open. 
\end{proposition}
\begin{proof}
The nonlinear map
$$
N_t: \calX^{k+2}_{\mu, -\delta}(\Si \ti \RP; \isu(E, H_0)) \longrightarrow \calX^k_{\mu-2,  -\delta}(\Si \ti \RP; \isu(E, H_0)),
$$
is well-defined and smooth, and it linearization $\ML_{t,s}$ at $s$ is an isomorphism. The statement now follows from
the implicit function theorem. 
\end{proof}
\end{subsection}

\begin{subsection}{A priori estimates and closeness}
  To prove that $I$ is closed, we must show that if $H_j$ is a sequence of solutions corresponding to $t_j \in I$, and if $t_j \to \bar{t}$,
  then $H_j$ also has a limit.  The analytic steps are essentially the same as \cite{HeMazzeo2018} except for the initial step, which is
  the $\calC^0$ estimate, so we concentrate on this. 
	
  Denote by $\MC^{k,\al}_D(\Si \ti \RP)$ the space of sections which are uniformly in $\calC^{k,\alpha}$ on every strip $\si \ti [t, t+1]$,
  and which also vanish at $y=0$ (the subscript `D' stands for Dirichlet), and also set $\MC^{k,\al}_{D, -\delta} = e^{-y\delta}\calC^{k,\al}_D$.
  Fix $\chi \in \calC^{\infty}( \Si \ti \RP)$ with $\chi \geq 0$, $\chi(y) = 1$ for $y \geq 2$ and $\chi(y) = 0$ for $y \leq 1$.
  
\begin{proposition}{\cite[Proposition 8.1]{HeMazzeo2018}}
\label{scalarlaplacian}
If $\Delta$ is the scalar Laplacian, then
\begin{equation*}
\begin{split}
\Delta: \calC^{k+2,\al}_{D, -\delta}(\Si \ti \RP) \oplus \RR & \longrightarrow \calC^{k,\al}_{-\delta} (\Si \ti \RP) \\
(u,A) & \longmapsto \Delta u + A\Delta (\chi)
\end{split}
\end{equation*}
is an isomorphism.
\end{proposition}
	
We now obtain a $\MC^0$ estimate, cf.\ \cite{HeMazzeo2018}:
\begin{proposition}  If $s$ is a Hermitian endomorphism satisfying $N_t(s) = 0$, then there exist a constant $C$
depending only on $H_0$ such that 
\begin{equation}
|s|_{\calC^0(\Si\ti\RP)}\leq C.
\end{equation}
\label{C0estimate}
\end{proposition} 
\begin{proof}
Taking the inner product of \eqref{quasilinearformequation} with $s$, where $H=H_0e^s$, gives
\begin{equation}
\Delta|s|^2+|v(s)\na s|^2+t|s|^2+\langle \Omega_{H_0},s\rangle =0;
\label{ausefulformofmethodofcontinouationequation}
\end{equation}
here $\Delta= -\del_y^2 + \Delta_{\Sigma}$ and $|v(s)\na s|$ is as in \eqref{vsdefinition}. Let $M:=\sup |s|$, then 
\[
\Delta |s|^2+t|s|^2\leq -\langle \Omega_{H_0},s\rangle  \Longrightarrow  \Delta|s|^2\leq M|\Omega_{H_0}|.
\]
By Proposition \ref{scalarlaplacian}, there exists $u \in \calC^{2,\alpha}_{D, -\delta}$ and $A \in \RR$ such that
$\Delta (u-A\chi)=|\Omega_{H_0}|$, hence $\Delta(|s|^2-Mu+AM\chi)\leq 0$, i.e., $|s|^2 - Mu + AM$ is a subsolution.
Since both $s$ and $u$ decay as $y \to \infty$ and vanish at $y=0$, and $\chi$ is bounded, we conclude that $|s|^2 -Mu + AM\chi \leq 0$.
This gives that $|s|^2 \leq M^2 \leq  M \sup (|u| + |A|)$, which gives the desired bound since $u$ and $A$ depend only on $\Omega_{H_0}$. 
\end{proof}
	
\begin{theorem}
\label{interiorckbound}
Let $N_t(s) = 0$, and suppose that $H_0$ satisfies the tilted Nahm pole boundary condition. Then for any $k \in \mathbb N$ and $\al\in(0,1)$, 
\[
[s]_{y^{2}\MC_0^{k,\al}}\leq C,
\]
where C depends only on $k,\al$ and $\Omega_{H_0}$, but not on $t$. 
\end{theorem}
The decay estimate is exactly the same as in \cite{HeMazzeo2018}, and leads to
\begin{proposition}
Assuming that $||s||_{L^\infty} + ||e^{-\delta y}\Omega_{H_0}||_{\calC^k}\leq C_k$ for any $k \geq 0$,  then for all $k$, 
$||e^{-\delta y}s||_{\calC^k}\leq C_k'$. 
\end{proposition}

\begin{theorem}
\label{fullestimate}
Suppose that $N_t(s) = 0$ and $H_0$ has a Nahm pole but no knot singularities. Let $\kappa$ be the first positive 
indicial root of $\ML_{H_0}$.  Then for all $k \in \mathbb N$, there is an a priori estimate 
\[
[s]_{\calX^k_{\kappa', \delta}}\leq C
\]
for any $0 < \kappa' < \kappa$, where $C$ depends on $k,l,\al$ and $\Omega_{H_0}$, but not on $t$. 
\end{theorem}	

An immediate corollary is 
\begin{corollary}
	$I$ is closed in $[0,1]$. 
\end{corollary}

\begin{theorem}
	\label{Thm_surjective}
The maps $I_{\Opers}^{\beta}:\MM^{\beta}_{\mathrm{TBE}}\to\MM^{\beta}_{\Op}$ is surjective.
\end{theorem}
\end{subsection}
\end{section}





	

\begin{section}{Uniqueness and Properness}
\subsection{Uniqueness}
Uniqueness of the solution is proved using convexity of the Donaldson functional.  For any two Hermitian metrics $K$ and $H=Ke^{s}$,
with $\Tr(s)=0$, write 
\begin{equation}
\Omega_{H,K}:=\frac{i}{2}\Lambda(\cos^2 \beta[\MD_1,\MD_1^{\da}]-\sin^2\beta[\MD_2,\MD_2^{\da}])+[\MD_3,\MD_3^{\da}],
\end{equation} 
where $\MD_i^{\da}$ is the conjugate with respect to $H$ defined in Section 2; the subscript $K$ emphasizes that when $K$ is fixed,
$\Omega_{H,K}=0$ is an equation for $s$. 
	
Define a Donaldson functional for the twisted Bogomolny equations in analogy with the well-known Donaldson functional for the
Hermitian-Yang-Mills equations \cite{donaldson1985anti, Donaldson1987Infinite, Simpson1988Construction}:
\begin{equation}
\MM(H,K)=\int_0^1\int_{\Sigma\ti\RP}\langle s,\Omega(Ke^{us},K)\rangle _K\ \omega\we dy \we du, 
\end{equation}
where $\omega$ is the volume form of $\Si$.  This functional reveals the variational structure for the \EBE. Indeed, 
setting $H_t=Ke^{ts}$, 
\begin{equation}
\begin{split}
\frac{d}{dt}\MM(H_t,K)=&\int_{\Sigma\ti\RP}\Tr(\Omega_{H_t,K}s)\omega\we dy, \\
\frac{d^2}{dt^2}\MM(H_t,K)=&\int_{\Si\ti\RP}\cos^2\beta|\MD_1s|^2+\sin^2\beta|\MD_2s|^2+|\MD_3 s|^2\\
&+\int_{\Si\ti\RP} \cos^2\beta\bar{\partial}\Tr(\MD_1^{\da}s\we s)+\sin^2\beta\pa\Tr(\MD_2^{\da}s\we s) +\partial_y\Tr(\MD_3^{\da}s\we s).
\label{Donaldsonfunctionalvariation}
\end{split}
\end{equation}
	
We now use this to prove injectivity of the maps $I_{\mathrm{Opers}}^{\beta}$ 
\begin{proposition}
	\label{Prop_injective}
Given any element in $\MOP^{\beta}$, suppose $H,K$ are two solutions to the twisted Bogomolny equations 
with the same singularity type and corresponding to this same set of holomorphic data. Then $H=K$.
\end{proposition}
\begin{proof}   Write $H = K e^s$ and $H_t=Ke^{ts}$.  By the indicial root computations for $\ML$, the order of vanishing of $s$ in $y$ is
  greater than $1$, hence the boundary terms in \eqref{Donaldsonfunctionalvariation} vanish.  Furthermore, the Higgs pair associated to
  $(\MD_1,\MD_2)$ is stable, so $\Ker\;\MD_1\cap \Ker\;\MD_2=\emptyset$. Hence if we set $m(t):=\MM(H_t,K)$, then
  $m'(0)=0$ and $m'' > 0$ if $s\not\equiv 0$. However, since $m(0) = m(1) = 0$, we see that $m \equiv 0$, so $H \equiv K$ after all. 
\end{proof}
	
\begin{corollary}
The maps $I_{\Opers}^{\beta}:\MM^{\beta}_{\mathrm{TBE}}\to\MM^{\beta}_{\Opers}$ is injective.
\end{corollary}
	
We have now established the main result,  that the maps $I_{\Opers}^\beta$ is a bijection.

\subsection{Properness}
We finally consider the properness of the maps $I^{\beta}_{\Opers}$.  We first define the topologies on
$\MOP^{\beta}$ and $\MM^{\beta}_{\TBE}$. 
	
By Theorem \ref{ClassificationofOpers}, we can write $\MOP^{\beta}=\{\na_{\mathbf{q}}\}$, where
$\mathbf{q}=(q_2,\cdots,q_n)\in \oplus_{i=2}^n H^{0}(K^i)$. Fix a metric on $K$ and use the $\MC^0$ norm on $\oplus_{i=2}^n H^{0}(K^i)$
to define the topology of $\MOP^{\beta}$. The moduli space $\MM^{\beta}_{\TBE}$ consists of pairs $\mathscr{A}:=(A,\phi,\phi_1)$ satisfying
the tilted Nahm pole boundary conditions. If $\msA_1$ and $\msA_2$ are two solutions, define 
$$
\|\msA_1-\msA_2\|_{\msN}:=\{\sup_k\|\msA_1-\msA_2\|_{\calX^k_{\mu, -\delta}}<\infty\ \mbox{for all}\ k\},
$$
for any fixed $\mu\in (-1,1)$ and $\delta\in(0,1)$. 

\begin{proposition}
\label{Prop_proper}
If $X\subset \MM_{\Opers}^{\beta}$ is a compact subset, set $Y:=(I^{\beta}_{\Opers})^{-1}X$. Then any sequence $\{y_n\}$ 
in $Y$ has a convergent subsequence $\{y_{n_k}\}$.
\end{proposition}
\begin{proof}
Set $x_n = I^\beta_{\Opers}(y_n) \in X$. It is possible to construct model approximate solutions $\Omega_n$ uniformly over
compact subsets of $X$. Let us also fix a solution $\msA_0$ to 
the twisted Bogomolny equations. Since $X$ is compact, then by Theorem \ref{existadmissble1}, we have $|\Omega_n|_{C^0}\leq C$
  where $C$ is independent of $n$. By Theorem \ref{fullestimate}, we obtain that $\|y_n-\msA_0\|_{\calX^k_{\mu, -\delta}}\leq C_k$ for
  any $k\in\mathbb{N}$. Hence there exists a subsequence $y_{n_i}-\msA_0$ such that $y_{n_i}-\msA_0$ converges in the norm
  $|| \cdot ||_\msN$, so $y_{n_i}$ also convergent.
\end{proof}
	
\begin{theorem}
The map $I^{\beta}_{\Opers}:\MM_{\TBE}^{\beta}\to \MM_{\Opers}^{\beta}$ is a diffeomorphism.
\end{theorem}
\begin{proof}
Implicit in the discussion above is the fact that the spaces $\MM_{\TBE}^\beta$ and $\MM_{\Opers}^\beta$ are both smooth manifolds.
The properness of $I^{\beta}_{\Opers}$ is Proposition \ref{Prop_proper}, and it is bijective by Theorem \ref{Thm_surjective}
and Proposition \ref{Prop_injective}.  Following through the construction in Section \ref{Sec_KobayashiHitchinMap}, we see that
$I^{\beta}_{\Opers}$ is not only  continuous, but actually a diffeomorphism.
\end{proof}

\end{section}

\section*{Appendix:The twisted Bogomolny equations}
In this Appendix, we discuss the tilted Nahm pole boundary condition for the twisted Kapustin-Witten equations, as well as the
Hermitian-Yang-Mills structure for its reduction, the twisted Bogomolny equations. We refer to
\cite{gaiotto2012knot,mikhaylov2018teichmuller} for more detailed explanations.
	
\subsection*{Tilted Nahm pole boundary condition}
Let $P$ denote a $G$-bundle over $M^4$, and $A$ a connection and $P$ a $1$-form over $P$. Then the twisted Kapustin-Witten
equations \cite{KapustinWitten2006} are
\begin{equation}
F_A-\Phi\we\Phi+\frac{t-t^{-1}}{2}d_A\Phi+\frac{t+t^{-1}}{2}\st d_A\Phi=0,\;d_A\st \Phi=0,
\end{equation}
	
We focus on the setting where $M=X\ti\RP_y$ where $X$ is a 3-manifold and $\RP_y = (0,\infty)$ with coordinate $y$.
Let $\mathfrak{g}$ be the Lie algebra of $G$.  We first consider the boundary condition on $X\times \{0\}\subset M$.
Given a principle embedding $\rho:\mathfrak{su}(2)\to\mathfrak{g}\otimes \mathbb{C}$,
given $x\in X$, let $\{\mathfrak{e}_a\}_{a=1,2,3}$ be an orthonormal basis of $T^*X$ and $\{\mathfrak{t}_a\}$ sections of 
the adjoint bundle $\gpp$ lying in the image of $\rho$ such that $[\mathfrak{t}_a,\mathfrak{t}_b]=\epsilon_{abc}\mathfrak{t}_c$.
We write the dreibein form $e:=\sum_{i=1}^3 \mft_i \mfe_i$, so $e$ gives an endomorphism $TX \to \gpp$. The definition of $e$
depends on the choice of $\rho.$  
	
\begin{definition}
For each $t=\tan(\frac{\pi}{4}-\frac{3}{2}\beta)$, a solution $(A,\Phi)$ to \eqref{TKW} over $M$ is a tilted Nahm pole solution
if for any point $x\in X$, there exist $\{\mathfrak{e}_a\}$, $\{\mathfrak{t}_a\}$ as above such that
 \begin{equation}
A=\frac{e}{y}\sin \beta+\MO(y^{-1+\ep}),\;\Phi=\frac{e}{y}\cos \beta+\MO(y^{-1+\ep}),\ \ \mbox{as}\ y \to 0,
\end{equation}
for some constant $\ep>0$.
\end{definition}
\begin{remark}
  For each $t$ there exist three possible corresponding values of $\beta$, and hence three different boundary conditions. For example,
  when $t=1$, we have $A\sim O(1),\Phi\sim \frac{e}{y}$, but the other two possibilities are
  $A\sim \pm\frac{\sqrt{3}}{2}\frac{e}{y},\;\Phi\sim\frac{1}{2}\frac{e}{y}$.
\end{remark}

\subsection*{The Dimensional Reduction}
Now consider the $4$-manifold $\mathbb{R}_{x_1}\times\Si_z\ti \RP_y $, with coordinates $(x_1,z,y)$. We write
$\hA=A+A_1dx_1$ and $\hP=\phi+\phi_1dx_1$ and consider $\mathbb{R}_{x_1}$ invariant solutions.  Fixing the orientation
$dx_1\we d\Si\we dy$, the twisted Kapustin-Witten equations reduce to 
\begin{equation}
\begin{split}
&F_A-\phi\we\phi+\frac{t-t^{-1}}{2}d_A\phi-\frac{t+t^{-1}}{2}(\st d_A\phi_1+\st[\phi,A_1])=0,\\
&d_AA_1-[\phi,\phi_1]+\frac{t-t^{-1}}{2}(d_A\phi_1+[\phi,A_1])-\frac{t+t^{-1}}{2}\st d_A\phi=0,\\
&d_A^{\st}\phi-[\phi_1,A_1]=0,
\label{GEBE}
\end{split}
\end{equation}
where $\st$ is the Hodge star operator on $\Si\ti\RP$. When $t=1$ and $A_1 = 0$, this recovers the previous extended Bogomolny equations and
there is a linear relationship between $\phi_1$ and $A_1$.
	
Introduce the condition $A_1-\tan \beta \phi_1=0$ and write $d_A=D_z+D_{\bz}+D_y$.  Using the local coordinate $\Si_z\ti\RP_y$,
the first equation in \eqref{GEBE} becomes
\begin{equation}
\begin{split}
&F_{z\bz}-[\phi_z,\phi_{\bz}]+\frac{t-t^{-1}}{2}(D_z\phi_{\bz}-D_{\bz}\phi_z)-\frac{t+t^{-1}}{2}\frac{i}{2}D_y\phi_1=0,\\
&F_{y\bz}+\frac{t-t^{-1}}{2}D_y\phi_{\bz}-\frac{t+t^{-1}}{2}i(D_{\bz}\phi_1+[\phi_{\bz},A_1])=0,\\
\label{1ndcoo}
\end{split}
\end{equation}
	
We write the second equation in \eqref{GEBE} in local coordinates:
\begin{equation}
\begin{split}
&\pa_y A_1+\frac{t-t^{-1}}{2}D_y\phi_1-\frac{t+t^{-1}}{2}(-2i)(D_z\phi_{\bz}-D_{\bz}\phi_z)=0\\
&D_{\bz}A_1-[\phi_{\bz},\phi_1]+\frac{t-t^{-1}}{2}(D_{\bz}\phi_1+[\phi_{\bz},A_1]-\frac{t+t^{-1}}{2}iD_y\phi_{\bz})=0
\label{2ndcoo}
\end{split}
\end{equation}
	
Finally, the third equation in \eqref{GEBE} becomes 
\begin{equation}
D_{\bz}\phi_z+D_z\phi_{\bz}=0.
\end{equation}

Setting $t=\tan(\frac{\pi}{4}-\frac{3}{2}\beta)$, we compute that $\frac{t-t^{-1}}{2}=-\tan(3\beta),\;\frac{t+t^{-1}}{2}=\frac{1}{\cos(3\beta)}$.
Set $f:=F_{A}-\phi\we\phi$; then taking an appropriate linear combination of the previous equations, we get 
\begin{equation}
\begin{split}
f_{z\bz}&=-\cot 2\beta (D_z\phi_{\bz}-D_{\bz}\phi_z),\\
iD_y\phi_1&=-\frac{1}{\sin\beta}(D_z\phi_{\bz}-D_{\bz}\phi_z)\\
iF_{yz}&=-\frac{\cos 2\beta}{\cos \beta}D_z\phi_1+2\sin\beta[\phi_z,\phi_1]\\
iD_y\phi_z&=2\sin\beta D_z\phi_1+\frac{\cos 2\beta}{\cos \beta}[\phi_z,\phi_1]\\
0&=D_z\phi_{\bz}+D_{\bz}\phi_z.
\label{equtions}
\end{split}
\end{equation}
	
We also impose the asymptotic boundary condition $\phi_1 \to 0$ as $y\to\infty$. This system still reduces to the Hitchin equation in that
limit, and hence determines a flat $SL(n,\mathbb{C})$ connection.
	
\begin{remark}
Any $t$ corresponds to three different values of $\beta$.  Since the vanishing condition $A_1-\tan \beta \phi_1=0$
  depends on $\beta$, there are actually three different equations \eqref{equtions} corresponding to the same value of $t$. 
\end{remark}
	
Now define the covariant derivatives
\begin{equation}
\begin{split}
\MD_1&=D_{\bz}-\phi_{\bz}\tan\beta\\
\MD_2&=D_z+\phi_z\cot \beta\\
\MD_3&=D_y-i\frac{\phi_1}{\cos \beta}
\end{split}
\label{threeoperators}
\end{equation}
We compute their commutators and the moment map equations:
\begin{equation}
\begin{split}
I_{12}=&-\sin 2\beta[\MD_1,\MD_2]=\sin 2\beta f_{z\bz}+\cos 2\beta(D_z\phi_{\bz}-D_{\bz}\phi_z)-(D_z\phi_{\bz}+D_{\bz}\phi_z),\;\\
I_{13}=&[\MD_1,\MD_3]=F_{y\bz}-\tan\beta D_y\phi_z-\frac{i}{\cos\beta}D_z\phi_1+\frac{i\sin\beta}{\cos^2 \beta}[\phi_z,\phi_1],\;\\
I_{23}=&[\MD_2,\MD_3]=F_{yz}+\cot \beta D_y\phi_z+\frac{i}{\cos \beta} D_z\phi_1+\frac{i}{\sin \beta}[\phi_z,\phi_1],\;\\
I_{mm}=&-(\cos^2 \beta[\MD_1,\MD_1^{\da}]+\sin^2\beta[\MD_2,\MD_2^{\da}]+\frac{1}{4}[\MD_3,\MD_3^{\da}])\\
=&\cos 2\beta f_{z\bz}-\sin 2\beta(\MD_z\phi_{\bz}-\MD_{\bz}\phi_z)-\frac{i}{2\cos \beta}D_y\phi_1.
\label{system}
\end{split}
\end{equation}
It is straight forward to check that \eqref{system} and \eqref{equtions} are equivalent. 
	
\begin{corollary}
If $\phi_1=0$, then the $y$-independent solution of \eqref{system} satisfies the Hitchin equations.
\label{boundaryconditioninfinity}
\end{corollary}
\begin{proof}
By \eqref{equtions}, if $\phi_1=0$, the $y$-independent solution satisfies
$$
f_{z\bz}=-\cot(2\beta)(D_z\phi_{\bz}-D_{\bz}\phi_z),\;D_z\phi_{\bz}-D_{\bz}\phi_z=0,\;D_z\phi_{\bz}+D_{\bz}\phi_z=0,
$$ 
which is equivalent to the Hitchin system \eqref{Hitchinequation}.
\end{proof}
	
\bibliographystyle{plain}
\bibliography{references}
\end{document}